\documentclass[11pt]{amsart}

\RequirePackage{amssymb}
\RequirePackage{enumerate}
\RequirePackage{amsfonts}
\RequirePackage{amsmath}
\RequirePackage{mathrsfs}
\RequirePackage{dsfont}
\allowdisplaybreaks[1]

\newcounter{fig}

\newcommand{\mybic}{\author{Gianluca Cassese}
                     \address{Universit\`{a} Milano Bicocca}
                     \email{gianluca.cassese@unimib.it}
                     \curraddr{Department of Economics, Statistics and Management 
                                  Building U7, Room 2097, via Bicocca 
                                  degli Arcimboldi 8, 20126 Milano - Italy}}

\newcommand{\Sim}{\mathscr{S}} 
 
\DeclareMathOperator*{\co}{co}

\DeclareMathOperator*{\arginf}{arginf}

\newcommand{\Ring}{\mathscr R}

\newcommand{\B}{\mathfrak{B}} 
\newcommand{\A}{\mathscr{A}} 
\newcommand{\F}{\mathscr{F}}
\newcommand{\R}{\mathbb{R}} 
\newcommand{\N}{\mathbb{N}} 

\newcommand{\Q}{\mathbb{Q}}
\newcommand{\Bor}{\mathscr{B}}
\newcommand{\RR}{\overline{\R}}

\newcommand{\quot}[1]{\textit{``#1''}}
\newcommand{\qtext}[1]{\quad\text{#1}\quad}

\newcommand{\abs}[1]{\vert #1\vert} 
\newcommand{\babs}[1]{\big\vert #1\big\vert}

\newcommand{\restr}[2]{#1\vert #2}

\newcommand{\net}[3]{\langle #1_{#2}\rangle_{#2\in #3} } 
\newcommand{\neta}[1]{\net{#1}{\alpha}{\mathfrak A}} 
\newcommand{\nnet}[3]{\langle #1\rangle_{#2\in #3} } 
 
\newcommand{\seq}[2]{\net{#1}{#2}{\mathbb{N}}} 
\newcommand{\sseq}[2]{\nnet{#1}{#2}{\mathbb{N}}} 
\newcommand{\seqn}[1]{\seq{#1}{n}} 
\newcommand{\sseqn}[1]{\sseq{#1}{n}}

\newcommand{\norm}[1]{\Vert #1\Vert}

\newcommand{\set}[1]{\mathds{1}_{#1}}
\newcommand{\sset}[1]{\mathds{1}_{\{#1\}}}

\newcommand{\cl}[2][ ]{\overline{#2}^{\ #1}}
\newcommand{\cco}[1][ ]{\overline\co^{#1}}

\newcommand{\Cts}[2][]{\mathscr{C}_{#1}(#2)}

\newcommand{\Lin}[1]{\mathfrak{L}(#1)}
\newcommand{\bLin}[1]{\mathfrak{L}\big(#1\big)}
\newcommand{\Fun}[1]{\mathfrak{F}(#1)}
\newcommand{\bFun}[1]{\mathfrak{F}\big(#1\big)}

\newcommand{\emp}{\varnothing}
\newcommand{\0}{\emptyset}

\newcommand{\iref}[1]{(\textit{\ref{#1}})}
\newcommand{\imply}[2]{\iref{#1}$\Rightarrow$\iref{#2}}

\newcommand{\tiref}[1]{(\textit{#1})}
\newcommand{\timply}[2]{(\textit{#1})$\Rightarrow$(\textit{#2})}

\newtheorem{theorem}{Theorem}
\theoremstyle{plain}

\newtheorem{corollary}{Corollary}

\newtheorem{definition}{Definition}
\newtheorem{example}{Example}

\newtheorem{lemma}{Lemma}

\newcommand {\Neg}{\mathscr N}

\newcommand {\Ha}{\mathscr H}
\newcommand {\Meas}{\mathscr M}
\DeclareMathOperator*{\con}{con} 
\newcommand{\ccon}[1][]{\overline{\con}^{#1}}

\begin{document} 
\title[Conglomerability]
{Conglomerability and the Representation of 
Linear Functionals}
\mybic
\date \today 

\begin{abstract} 
We prove results concerning the representation 
of linear functionals as integrals of a given random
 quantity $X$. The existence of such representation 
is related to the notion of conglomerability, originally 
introduced by de Finetti and Dubins. We show that 
this property has interesting applications in probability 
and in analysis. These include a version of the 
extremal representation theorem of Choquet, a proof 
of Skorohod theorem and of the statement that 
Brownian motion assumes whatever family of finite 
dimensional distributions upon a change of the 
probability measure. 
\end{abstract}

\maketitle

\section{Introduction.}
In this paper we study the classical problem of 
the integral representation of linear functionals 
with a degree of generality which does not permit
the direct application of classical techniques. 
Conglomerability is then necessary and sufficient 
to conveniently transform the original problem into 
one in which integral representation is indirectly
possible.  Although the fields in which our results 
may be fruitfully applied are disparate, we were 
motivated by the problem of 
{\it existence of companions} 
that arises in several places in probability and 
statistics.

Let $S$ and $\Omega$ be 
given, non empty sets, $\Ha$ a family of real 
valued functions on $S$ and $(X,m)$ a pair, with
$X$ a mapping of $\Omega$ into $S$ and $m$
a positive, finitely additive set function $m$  
on $\Omega$. 
Following  Dubins and Savage \cite{dubins savage}, 
we say that a pair $(X',\mu)$ on a 
set $\Omega'$ is a {\it companion} to $(X,m)$, 
relatively to $\Ha$, if it solves the 
equation
\begin{equation}
\label{problem}
h(X')\in L^1(\mu)
\qtext{and}
\int h(X)dm
	=
\int h(X')d\mu
\qquad
h\in\Ha,
h(X)\in L^1(m).
\end{equation}
The collection $\Ha$ is interpreted as a model of 
the information available. 

Finding a correct statistical model $X'$ for a 
given data sample is a problem fitting into 
\eqref{problem}: set $S=\Omega=\R$, let $X$ 
be the identity, $m$ the sample distribution 
and each $h\in\Ha$ a statistic. Given a predictive 
marginal $m$ on an algebra $\A$, a similar problem 
in Bayesian statistics is that of finding a parametric 
family 
$\mathcal Q=\{Q_\theta:\theta\in\Theta\}$ 
of probabilities and a prior $\lambda$ on the 
parameter space $\Theta$ such that
\begin{equation}
\label{bayes}
m(A)=\int_\Theta Q_\theta(A)d\lambda
\qquad
A\in\A.
\end{equation}

Dubins \cite{dubins} proved long ago that the 
existence of a disintegration formula similar to 
\eqref{bayes} is equivalent to {\it conglomerability}, 
a notion originally due to de Finetti \cite{de finetti} 
that has remained undeservedly neglected 
outside a limited number of distinguished authors 
(which include Schervisch et al. \cite{SSK}, Hill and 
Lane \cite{hill lane} and Zame \cite{zame}). The 
conglomerability property, we believe, may be 
formulated in more general terms than those in 
which it was originally stated and it may be applied 
to more ambitious problems in probability and in 
analysis than those for which it had been originally 
devised. 

We found it useful to write problem \eqref{problem}
in more abstract terms, replacing the integral on left 
hand side with a linear functional and modelling the 
action of $X'$ on $\Ha$ as a linear transformation. 
We solve this version of our problem in Theorem 
\ref{th representation}, obtaining a special integral 
representation for a conglomerative linear functional 
on an arbitrary vector space. The absence of any 
structure on the underlying space, save linearity, 
makes the claim significantly more general than 
classical Riesz representation theorems. In this 
functional analytic formulation, conglomerability 
may be nicely restated as a geometric property. 
In Corollary \ref{cor choquet} we show that if 
$\Phi$ and $\Psi$ are two sets of positive, linear 
functionals on a vector lattice then $\Phi$ is 
$\Psi$-conglomerative if and only if each $\phi\in\Phi$ 
is the barycentre of a measure supported by 
$\Psi$. In Corollary \ref{cor choquet V} we obtain 
a generalization of the original theorem of Choquet 
\cite{choquet}. 

Theorem \ref{th representation} admits a large 
number of implications, the most immediate of 
which is the existence of companions with or 
without additional conditions on the representing 
measure $\mu$, such as countable additivity or 
absolute continuity with respect to some given, 
reference set function. An immediate corollary 
is that, relatively to continuous functions, a normally
distributed random quantity is companion to 
\textit{any} $X$ and that Brownian motion can 
assume whatever family of finite dimensional 
distributions on $\R$ upon an appropriate choice 
of the underlying probability. We also provide 
applications to the classical Skhorohod 
representation theorem in the case in which $S$ 
is separable. 

In the closing section we prove some results 
concerning the representation of convex functions 
as integrals. We show that any convex function 
on $\R$ decomposes into the sum of a piece 
wise linear component and an integral part, a 
representation curiously near to the one popular 
in mathematical finance as a model for option 
prices.

All proofs are quite simple and, despite a natural 
interest for countable additivity, they are obtained 
by exploiting the theory of the finitely additive 
integral in which the measurability constraint is 
much less burdensome. We hope to disprove thus, 
at least partially, the harsh judgement of Bourgin 
\cite[p. 173]{bourgin} that \quot{an integral 
representation theory based on finitely additive 
measures is virtually useless}.

\section{Notation and Preliminaries.}
Throughout the paper the symbol 
$\Fun{\Omega,S}$ (resp. $\Fun\Omega$) 
denotes the family of functions mapping 
$\Omega$ into $S$ (resp. into $\R$) and 
$\mathfrak F$ is replaced with $\mathfrak L$, 
$\mathscr C$ or $\mathscr C_K$ when 
restricting to linear, continuous or continuous 
functions with compact support, respectively. 
A collection $\{f_y:y\in Y\}\subset\Fun{X,S}$ is 
also written as a function $f\in\Fun{X\times Y,S}$ 
with $f(x,y)=f_y(x)$ and viceversa. If 
$f\in\Fun{\Omega,S}$ and $A\subset\Omega$ 
the symbols $\restr f A$ and $f[A]$ designate 
the restriction of $f$ to $A$ and the image of 
$A$ under $f$. A subset $\Ha$ of $\Fun S$ 
is Stonean if $h\in\Ha$ implies $h\wedge 1\in\Ha$, 
where $1\in\Fun{S}$ indicates the function 
constantly equal to $1$. 

If $\A$ is a ring of subsets of $\Omega$, then 
$\Sim(\A)$ and $\B(\A)$ denote the families of 
$\A$ simple functions and its closure in the 
topology of uniform convergence, respectively. 
$fa(\A)$ (resp. $fa(\Omega)$) is the space of real 
valued, finitely additive set functions on $\A$ 
(resp. $2^\Omega$) and $ba(\A)$ the subspace 
of set functions of bounded variation. To indicate
that $\A$ is a ring of subsets of $\Omega$ and 
that $\lambda\in fa(\A)_+$, i.e. that $(\A,\lambda)$
is a measure structure on $\Omega$ we write 
more compactly
$(\A,\lambda)
\in
\Meas(\Omega).$ 

We recall a few definitions and facts relative to 
the finitely additive integral (see \cite{rao} and 
\cite{bible}) given $(\A,\lambda)\in\Meas(\Omega)$%
\footnote{
To be formal, we depart from the classical theory 
of Dunford and Schwartz which has an extended 
real valued set function on an algebra of sets as its 
starting point. Our notion of a simple function is
obtained from theirs after restricting to the family 
of sets of finite measure, a ring, and coincides
therefore with the notion of {\it integrable simple 
functions} of Dunford and Schwartz. Thus, our 
notion of measurability is more restrictive than 
that of {\it total measurability} given in \cite[III.2.10]
{bible} although integrable functions are defined 
by Dunford and Schwartz as being measurable in 
our restrictive sense.
}. 
$X\in\Fun{\Omega}$ is $\lambda$-measurable 
if there exists a sequence $\seqn X$ in $\Sim(\A)$ 
that $\lambda$-converges to $X$, i.e. such that
\begin{equation}
\lim_n\lambda^*(\abs{X_n-X}>c)=0
\quad\text{for every}\quad
c>0	
\end{equation}
where the set function $\lambda^*$ and its conjugate 
$\lambda_*$ are defined (with the convention 
$\inf\emp=\infty$) as 
\begin{equation}
\lambda^*(E)=\inf_{\{A\in\A:E\subset A\}}\lambda(A)
\quad\text{and}\quad
\lambda_*(E)=\sup_{\{B\in\A:B\subset E\}}\lambda(B)
\qquad	
E\subset\Omega.
\end{equation}
If $S$ is a topological space, $X$ is $\lambda$-tight 
if for all $\varepsilon>0$ there exists $K\subset S$ 
compact such that 
$\lambda^*(X\notin K)<\varepsilon$. 
$X$ is $\lambda$-integrable, $X\in L^1(\lambda)$, 
if there is a sequence $\seqn X$ in $\Sim(\A)$ 
that $\lambda$-converges to $X$ and is Cauchy 
in $L^1(\lambda)$; we then write $\int Xd\lambda$ 
or $\int_\Omega X(\omega)d\lambda(\omega)$. 
We notice that if $A,B\subset\Omega$ and 
$f\in L^1(\lambda)$, then
\begin{equation}
\label{tchebycheff}
\set A\le f\le\set B
\quad\text{implies}\quad
\lambda^*(A)\le\int fd\lambda\le\lambda_*(B).
\end{equation}

The following collections are important:
\begin{subequations}
\begin{equation}
\label{D(X,l)}
D(X,\lambda)
	=
\big\{t>0:\lim_n\lambda_*(X>t-2^{-n})
	=
\lim_n\lambda_*(X>t+2^{-n})\big\},
\end{equation}
\begin{equation}
\label{R0(X,l)}
\Ring_0(X,\lambda)
	=
\big\{\{X>t\}:t\in D(X,\lambda)\big\}
\cup
\big\{\{-X>u\}:u\in D(-X,\lambda)\big\},
\end{equation}
\begin{equation}
\label{A(l)}
\A(\lambda)
	=
\big\{E\subset\Omega:\lambda^*(E)
=
\lambda_*(E)<\infty\big\}.
\end{equation}
\end{subequations}
There is clearly just one extension of $\lambda$ 
to $\A(\lambda)$ and $X$ is $\lambda$-measurable 
if and only if it is measurable with respect to such 
extension, which we shall denote, accordingly, 
again by $\lambda$. A sequence $\seqn X$ in 
$L^1(\lambda)$ converges to $X$ in norm if and 
only if it $\lambda$-converges to $X$ and is Cauchy 
in the norm of $L^1(\lambda)$, \cite[III.3.6]{bible}. 

We shall use the following results on measurability 
and integrability of a positive function.

\begin{lemma}
\label{lemma measurable}
Let $X\ge0$. $X$ is $\lambda$-measurable if and 
only if it is $\lambda$-tight and either 
(i) 
$\infty>\lambda_*(X>s)\ge\lambda^*(X\ge t)$ for 
all $0<s<t$,
(ii)
$\Ring_0(X,\lambda)\subset\A(\lambda)$,
or
(iii) 
the set $\big\{t>0:\{X>t\}\in\A(\lambda)\big\}$ is 
dense in $\R_+$.
\end{lemma}

\begin{proof}
If $X$ is $\lambda$-measurable it is $\lambda$-tight, 
\cite[p. 190]{karandikar jmva}. Choose $\seq Xk$ in 
$\Sim(\A)$ $\lambda$-convergent to $X$, fix
$s,\eta>0$ and $A_k^\eta\in\A$ such that 
$\{\abs{X-X_k}\ge\eta\}\subset A_k^\eta$ and 
$\lambda(A_k^\eta)
\le
\lambda^*(\abs{X-X_k}\ge\eta)+2^{-k}$.
\begin{align*}
\{X\ge s+2\eta\}
	\subset
\{X_k\ge s+\eta\}\cup A_k^\eta
	\subset
\{X>s\}\cup A_k^\eta
\end{align*}
so that
$\lambda^*(X\ge s+2\eta)
	\le
\lambda\big(\{X_k\ge s+\eta\}\cup A_k^\eta\big)
	\le
\lambda_*(X>s)+\lambda\big(A_k^\eta\big)
$
and $\lambda^*(X\ge s+2\eta)<\infty$. 
Assume \tiref{i}. If $t\in D(X,\lambda)$ then 
$\infty
	>
\lambda_*(X>t)
	=
\lim_n\lambda_*(X>t-2^{-n})
	\ge
\lambda^*(X\ge t)
	\ge
\lambda^*(X>t)$. 
\timply{ii}{iii} is obvious. Assuming \tiref{iii}, choose 
$\{0=t_0^n\le t^n_1\le
\ldots
\le t^n_{I_n}\le t^n_{I_n+1}=2^n\}$ 
such that $\{X>t^n_i\}\in\A(\lambda)$ for 
$i=1,\ldots,I_n$ and 
$\sup_{0\le i\le I_n}\abs{t^n_i-t^n_{i+1}}<2^{-n}$. 
Define
\begin{equation}
\label{Xn}
X_n
	=
\sum_{i=1}^{I_n-1}t^n_i\sset{t^n_i<X\le t^n_{i+1}}
	\in
\Sim\big(\A(\lambda)\big).
\end{equation}
Then 
$\{\abs{X-X_n}\ge2^{-n}\}
	\subset
\{ X>2^{n-1}\}$ so that $X_n$ $\lambda$-converges 
to $X$ whenever $X$ is $\lambda$-tight.
\end{proof}

\begin{lemma}
\label{lemma L1}
Let $X\ge0$.
$X\in L^1(\lambda)$ if and only if 
$
\int_0^\infty\lambda_*( X>t)dt
	=
\int_0^\infty\lambda^*( X>t)dt
	<
\infty
$.
Then,
\begin{equation}
\label{bernstein}
\int Xd\lambda
	=
\int_0^\infty\lambda_*(X>t)dt.
\end{equation}
\end{lemma}

\begin{proof}
Assume  
$\int\lambda_*( X>t)dt
	=
\int\lambda^*( X>t)dt<\infty$. Then $X$ is 
$\lambda$-tight and 
$\{t\in\R:\{ X>t\}\in\A(\lambda)\}$ is dense
in $\R_+$ so that $ X$ is $\lambda$-measurable. 
As in \eqref{Xn} we can construct an increasing 
sequence $\seqn X$ in $\Sim(\A(\lambda))$ such 
that $0\le X_n\le X$ and $\lambda$-converges 
to $X$. But then,
\begin{align}
\infty
	>
\int_0^\infty\lambda_*(X>t)dt
	\ge
\lim_n\int_0^\infty\lambda(X_n>t)dt
	=
\lim_n\int X_nd\lambda
	=
\int Xd\lambda
\end{align}
as $\seqn X$ is Cauchy in $L^1(\lambda)$. Assume 
conversely that $X\in L^1(\lambda)$ and take 
$b>a>\varepsilon>0$. If $\seqn X$ in $\Sim(\A)$ 
converges to $X$ in $L^1(\lambda)$, then
\begin{align*}
\int_{a+\varepsilon}^{b+\varepsilon}\lambda^*(X>t)dt
	&\le
\int_\pm^b\lambda(X_n>t)dt
	+
(b-a)\lambda^*(\abs{X-X_n}>\varepsilon)\\
	&\le
\int_{a-\varepsilon}^{b-\varepsilon}\lambda_*(X>t)dt
	+
2(b-a)\lambda^*(\abs{X-X_n}>\varepsilon)
\end{align*}
by \cite[3.2.8.(iii)]{rao}. Thus, 
$\int_\pm^b\lambda_*(X>t)dt
	=
\int_\pm^b\lambda^*(X>t)dt$ 
and
\begin{align*}
\int_\pm^b\lambda^*(X>t)dt
	=
\lim_n\int_\pm^b\lambda(X_n>t)dt
	=
\lim_n\int(b\wedge X_n-a)^+d\lambda
	=
\int(b\wedge X-a)^+d\lambda.
\end{align*}
Thus 
$
\label{inter+}
\int_0^\infty\lambda_*(X>t)dt
	=
\int_0^\infty\lambda^*(X>t)dt
	=
\int Xd\lambda
	<
	\infty
$
and  \eqref{bernstein} holds.
\end{proof}

Proving uniqueness of the set function generating 
a given class of integrals requires to identify a 
minimal element in $\Meas(\Omega)$ associated 
with a given family of functions. This we do by writing, 
for two $(\A,\lambda),(\Bor,\xi)\in\Meas(\Omega)$
\begin{equation}
\label{order}
(\A,\lambda)\preceq(\Bor,\xi)
\quad\text{whenever}\quad
\A\subset\Bor(\xi)
\quad\text{and}\quad
\restr{\xi}{\A}=\lambda.
\end{equation} 

\begin{lemma}
\label{lemma bochner}
Let $\Ha$ be a Stonean, convex cone in 
$\Fun{\Omega}_+$ and $\phi\in\Fun\Ha$. 
The family of those $(\A,\lambda)\in\Meas(\Omega)$
satisfying
\begin{equation}
\label{M(fi)}
\Ha\subset L^1(\lambda)
\quad\text{and}\quad
\int hd\lambda=\phi(h)
\qquad
h\in\Ha,
\end{equation}
is either empty or contains a minimal element, 
$(\Ring_\phi,\lambda_\phi)$. 
\end{lemma}

\begin{proof}
Assume that $(\A,\lambda)\in\Meas(\Omega)$ 
satisfies \eqref{M(fi)} and denote by $\Ring_\phi$ 
the smallest ring containing
\begin{equation}
\label{R0}
\Ring_{0,\phi}
	=
\big\{\{h>t\}:h\in\Ha,\ t\in D(h,\lambda)\big\}.
\end{equation}
Suppose that $(\Bor,\xi)$ is another such structure. 
Fix $h\in\Ha$ and consider the classical inequality
\begin{align}
\label{bochner}
\sset{h>a}
	\ge
\frac{h\wedge b-h\wedge a}{b-a}
	\ge
\sset{h\ge b}
\qquad
h\in\Ha, b>a>0.
\end{align}
As the inner term belongs to the linear span of 
$\Ha$, $\infty>\lambda_*(h>a)\ge\xi^*(h\ge b)$, 
by \eqref{tchebycheff}. Choosing $a$ and $b$ 
conveniently and interchanging $\lambda$ with 
$\xi$ we establish that $D(h,\lambda)=D(h,\xi)$ 
and that
\begin{align*}
\lambda^*(h\ge t)
=
\xi^*(h\ge t)
=
\xi_*(h>t)
=
\lambda_*(h>t)
\qquad
t\in D(h,\lambda).
\end{align*}
Thus,
$
\Ring_{0,\phi}\subset\Bor(\xi)
$
and $\lambda$ and $\xi$ coincide on $\Ring_{0,\phi}$
and therefore on the collection
\begin{align*}
\mathscr E
=
\big\{E\subset\Omega:
\set E\in\Sim(\Ring_{0,\phi})\big\}.
\end{align*}
To show that $\Ring_{0,\phi}$ is closed with 
respect to intersection, for $i=1,2$ pick $h_i\in\Ha$ 
and $t_i\in D(h_i,\lambda)$. Fix 
$t_1\wedge t_2\ge\eta>0$, define 
$h_\eta
=
\big(h_1-(t_1-\eta)\big)^+
\wedge
\big(h_2-(t_2-\eta)\big)^+$
and observe that
\begin{align*}
h_\eta
=
(h_1+h_2\wedge(t_2-\eta))
\wedge
(h_2+h_1\wedge(t_1-\eta))
-
(h_1\wedge(t_1-\eta)+h_2\wedge(t_2-\eta))
\in
\mathrm{span}(\Ha).
\end{align*}
Since the sets $D(h_\eta,\lambda)$ are dense in 
$\R_+$, choose 
\begin{align*}
\delta
	\in
(0,t_1\wedge t_2]\cap\Q\cap
\bigcap_{\eta\in\Q\cap(0,t_1\wedge t_2]}
D(h_\eta,\lambda).
\end{align*}
Then 
$\delta\in D(h_\delta,\lambda)$, 
$h_\delta\in\Ha$ 
and 
$\{h_1>t_1\}\cap\{h_2>t_2\}=\{h_\delta>\delta\}$. 
But then $\mathscr E$ too is closed with respect to
intersection and this fact together the linear structure 
of $\Sim(\Ring_{0,\phi})$ imply in turn that 
$\mathscr E$ is also closed with respect to set 
difference and, from
$\set{E_1\cup E_2}
	=
\set{E_1}+\set{E_2\backslash E_1}$,
to union as well. In other words, 
$\lambda$ and $\xi$ coincide on the ring 
$\mathscr E$ which contains $\Ring_{0,\phi}$
and {\it a fortiori} on $\Ring_\phi$. Let $h\in\Ha$, 
$t>s>0$ and 
$\lambda_\phi=\restr{\lambda}{\Ring_\phi}$. 
Then, $h$ is $\lambda_\phi$-tight because 
$h\in L^1(\lambda)$ and there are 
$t',s'\in D(h,\lambda)$ with $t>t'>s'>s$ and 
therefore such that
$\lambda_{\phi*}(h>s)
	\ge
\lambda_\phi(h>s')
	\ge
\lambda_\phi(h>t')
	\ge
\lambda_\phi^*(h\ge t)
$.
By Lemma \ref{lemma measurable} $h$ is thus 
$\lambda_\phi$-measurable and therefore 
$\int hd\lambda_\phi=\int hd\lambda$, by
\cite[II.8.1(e)]{bible}.
\end{proof}

Although the minimal structure 
$(\Ring_\phi,\lambda_\phi)$ 
will generally depend on $\phi$, the generated 
$\sigma$ ring corresponds to the usual notion, as 
$D(h,\lambda)$ is dense.

The next result, and its use in Theorem 
\ref{th representation}, provides the best illustration 
of our interest for set functions defined on {\it rings}.

\begin{lemma}
\label{lemma fubini}
Let $g\in\Fun{\Omega}_+$ be $\lambda$-measurable 
and define the ring 
$\Ring_g
=
\big\{A\in\A(\lambda):g\set A\in L^1(\lambda)\big\}$.
There exists a unique $\lambda_g\in fa(\Ring_g)_+$ 
such that
\begin{equation}
\label{lg}
\int f\lambda_g
	=
\int fgd\lambda
\qquad
f\in\B(\lambda),\ fg\in L^1(\lambda).
\end{equation}
\end{lemma}

\begin{proof}
\eqref{lg} implies 
$\lambda_g(A)
	=
\int\set Agd\lambda$ 
for every $A\in\Ring_g$ and thus uniqueness. In 
proving \eqref{lg} we may assume $f\in\B(\lambda)_+$. 
Let $\seqn f$ be an increasing sequence in 
$\Sim(\A(\lambda))$ such that $0\le f_n\le f$ and 
$f_n$ converges to $f$ uniformly, obtained as 
in \eqref{Xn}. Then $f_n$ is $\lambda$- and 
$\lambda_g$-convergent to $f$. Moreover, $f_n$ 
and $f_ng$ are Cauchy sequence in 
$L^1(\lambda_g)$ and $L^1(\lambda)$.
\end{proof}

\section{Integral Representation of Linear Functionals.}
\label{sec integral}

First we make the notion of conglomerability precise.
\begin{definition}
\label{def conglomerative}
Let $\Ha$ be a vector space. Then 
$\phi\in\Lin{\Ha}$ is said to be conglomerative 
with respect to $T\in\bFun{\Ha,\Fun{\Omega}}$ 
(or $T$-conglomerative) if $\phi(h)<0$ implies 
$\inf_\omega (Th)(\omega)<0$
for all $h\in\Ha$.
\end{definition}

$T$-conglomerative linear functionals form a 
convex cone in $\Lin\Ha$ which is $\Ha$-closed, 
i.e. closed in the topology induced by $\Ha$ on 
$\Lin\Ha$. 
Another key property is the following:

\begin{definition}
\label{def directed}
Let $\Ha$ be a vector space. A map
$T\in\bLin{\Ha,\Fun{\Omega}}$ is said to be
directed if:
\begin{equation}
\label{directed}
\forall h\in\Ha,\ \exists h'\in\Ha
\qtext{such that}
\abs{Th}\le Th'.
\end{equation}
\end{definition}

Proving property \eqref{directed} will be a delicate
step in most of the applications that follow. Two 
easy special cases are:
($\alpha$) 
when $\Ha$ is a vector lattice and $T$ is positive
and 
($\beta$)
when $T[\Ha]\subset\B(\Omega)$ and 
$\sup_h\inf_\omega(Th)(\omega)>0$ -- e.g. if 
$T[\Ha]$ contains the constants. In general, 
there are several important situations in which 
$\Ha$ is an ordered vector space but not a 
lattice. In such general situations a possibility 
is to restrict to the vector space
\begin{equation}
\label{HaT}
\Ha(T)
=
\big\{h\in\Ha:
\abs{Th}\le Th'\text{ for some }h'\in\Ha\big\}
\end{equation}
on which $T$ is directed, by construction.

Most results in this paper follow from the next 
claim. 

\begin{theorem}
\label{th representation}
Let $\Ha$ be a vector space and let 
$T\in\bLin{\Ha,\Fun{\Omega}}$ be directed. 
Write
$
L
	=
\big\{f\in\Fun{\Omega}:\abs f\le Th
\text{ for some }
h\in\Ha\big\}
$.
Then $\phi\in\Lin\Ha$ is $T$-conglomerative if 
and only if there exist
(i)
$F^\perp\in\Lin L_+$ with 
$F^\perp[L\cap\B(\Omega)]=\{0\}$ and
(ii) 
$(\Ring,\mu)\in\Meas(\Omega)$
such that 
\begin{equation}
\label{representation T}
L\subset L^1(\mu)
\qtext{and}
\phi(h)
	=
F^\perp\big(Th\big)+\int Thd\mu
\qquad
h\in\Ha.
\end{equation}
Moreover,
\begin{enumerate}[(a)]
\item
$\norm\mu=1$ 
if and only if
$\inf_\omega(Th)(\omega)\le\phi(h)$
for all $h\in\Ha$,
\item
$\mu$ may be chosen to be countably additive if 
$
\limsup_n\phi(h_n)\le0
$
for all sequences $\sseqn{(h_n,f_n)}$ in $\Ha\times L$
satisfying 
\begin{equation}
\label{ca}
1\ge f_n\downarrow0
\qtext{and}
\limsup_n\ \sup
\big\{\phi(g):g\in\Ha,\ Tg\le\abs{f_n-Th_n}\big\}
\le
0,
\end{equation}
\item
for each $L_0\subset L\cap\B(\Omega)$,
$\mu$ may be chosen to be $L_0$-maximal, 
i.e. maximal as a map on $L_0$.
\end{enumerate}
\end{theorem}

\begin{proof}
$T[\Ha]$ is a majorizing subspace of the vector lattice 
$L$, by \eqref{directed}. If $\phi$ is 
$T$-conglomerative
\begin{equation}
\label{F}
F\big( T h\big)
	=
\phi(h)
\qquad
h\in\Ha
\end{equation}
implicitly defines a positive linear functional $F$ on 
$T[\Ha]$. By \cite[theorem 1.32]{aliprantis}, $F$ 
extends as a positive linear functional (still denoted 
by $F$) to the whole of $L$. For each 
$\alpha\subset\Ha$ 
finite, let $h_\alpha\in\Ha$
be such that 
$Th_\alpha\ge\bigvee_{h\in\alpha}\abs{Th}$,
 $\Omega_\alpha=\{Th_\alpha\ne0\}$ 
 and define 
$I_\alpha\in\bFun{L,\Fun{\Omega_\alpha}}$ 
by letting 
\begin{align*}
I_\alpha(f)(\omega)
=
\frac{f(\omega)}{T h_\alpha(\omega)}
\qquad
f\in L,\ \omega\in\Omega_\alpha.
\end{align*}
Let also
\begin{equation}
L_\alpha
	=
\{f\in L:\abs f\le c\ Th_\alpha\text{ for some }c>0\}
\quad\text{and}\quad
H_\alpha=I_\alpha[L_\alpha].
\end{equation}
$H_\alpha$ is a sublattice of $\B(\Omega_\alpha)$ 
containing the constants; $f,g\in L_\alpha$ and 
$I_\alpha(f)\ge I_\alpha(g)$ 
imply $f\ge g$. Thus, upon writing
\begin{equation}
U_\alpha\big(I_\alpha(f)\big)=F(f)
\qquad
f\in L_\alpha
\end{equation}
we obtain yet another positive, linear functional 
$U_\alpha$ on $H_\alpha$. \cite[Theorem 1]{JMAA} 
implies 
\begin{equation}
U_\alpha\big(I_\alpha(f)\big)
	=
\int I_\alpha(f)d\bar m_\alpha
\qquad
f\in L_\alpha
\end{equation}
for some 
$\bar m_\alpha\in ba(\Omega_\alpha)_+$. 
Let 
$m_\alpha(A)
	=
\bar m_\alpha(A\cap\Omega_\alpha)$ 
for each $A\subset\Omega$. By Lemma 
\ref{lemma fubini}, we can write (with the 
convention $0/0=0$)
\begin{equation}
\label{mua}
F(f)
	=
\int\frac{f}{Th_\alpha}\set{\Omega_\alpha}dm_\alpha
	=
\int fd\bar\mu_\alpha
\qquad
f\in L_\alpha\cap\B(\Omega)
\end{equation}
with $\bar\mu_\alpha=m_{\alpha,g}$ defined as in 
\eqref{lg} with $g=\set{\Omega_\alpha}/Th_\alpha$. 
Since $L_\alpha\cap\B(\Omega)$ is a Stonean 
lattice, we deduce from Lemma \ref{lemma bochner} 
the existence of a minimal 
$(\Ring_\alpha,\mu_\alpha)\in\Meas(\Omega)$ 
supporting the 
representation \eqref{mua}. Define 
$\Ring=\bigcup_\alpha\Ring_\alpha$ and
$
\mu(A)
	=
\lim_\alpha\mu_\alpha(A)
$
for all $A\in\Ring$. $\alpha\subset\alpha'$ implies 
$L_\alpha\subset L_{\alpha'}$,
$(\Ring_\alpha,\mu_\alpha)
\preceq
(\Ring_{\alpha'},\mu_{\alpha'})$
as well as the martingale restriction
\begin{equation}
\label{martingale}
\mu_\alpha
	=
\restr{\mu_{\alpha'}}{\Ring_\alpha}
	=
\restr{\mu}{\Ring_\alpha}
\qquad
\alpha\subset\alpha'.
\end{equation}
But then for each $f\in L_\alpha$ with $f\ge0$,
\begin{equation}
\label{F split}
\begin{split}
F(f)
	&=
\lim_kF(f\wedge k)+\lim_kF\big((f-k)^+\big)\\
	&=
\lim_k\int(f\wedge k)d\mu+F^\perp(f)\\
	&=
\int fd\mu+F^\perp(f)
\end{split}
\end{equation}where we have set 
$F^\perp(f)=\lim_kF\big((f-k)^+\big)$ 
and the inequality 
$\mu^*(f>k)
\le 
k^{-1}\int f\wedge kd\mu
\le 
k^{-1}F(f)$
induces the conclusion that $f\wedge k$ is 
$\mu$-convergent to $f$ and is Cauchy in 
$L^1(\mu)$. $\int\abs{f}d\mu\le F(\abs{f})$ 
follows from \eqref{F split} and implies 
$L\subset L^1(\mu)$. \eqref{representation T} 
is a consequence of \eqref{F} and \eqref{F split}. 
Necessity is obvious as the right hand side of 
\eqref{representation T} defines a positive 
linear functional on $L$.

\tiref{a}.	
Suppose that $\phi(h)<a<\inf_\omega Th(\omega)$ 
for some $a\in\R$ and $h\in\Ha$. Then, by 
\eqref{representation T} and properties of 
$F^\perp$, $a>\int Thd\mu\ge a\norm\mu$ 
which is contradictory if $\norm\mu=1$.
Conversely, define $\hat\phi\in\Lin{\R\times\Ha}$ 
and $\hat T\in\bLin{\R\times\Ha,\Fun\Omega}$ 
implicitly by letting
\begin{equation}
\hat\phi(r,h)=r+\phi(h)
\qtext{and}
\hat T(r,h)=r+T(h)
\qquad
(r,h)\in\R\times\Ha.
\end{equation}
By assumption $\hat\phi$ is $\hat T$-conglomerative
and thus admits a pair $\hat F^\perp$ and $\hat\mu$ 
as above. Therefore
\begin{equation}
r+\phi(h)
=
\hat F^\perp(r+Th)+\int(r+Th)d\hat\mu
=
\hat F^\perp(Th)+\int(r+Th)d\hat\mu
\qquad
(r,h)\in\R\times\Ha.
\end{equation}
Letting $h=\0$ we deduce $\norm{\hat\mu}=1$ 
and, from this, 
$\phi(h)
	=
\hat F^\perp(Th)+\int Thd\hat\mu$ 
for every $h\in\Ha$.

\tiref{b}. 
Fix a sequence $\seqn f$ as in 
\eqref{ca}. By \cite[theorem 1.33]{aliprantis} 
the extension of $F$ from $T[\Ha]$ to $L$  
constructed above may be chosen such 
that $\inf_{h\in\Ha}F\big(\abs{f-Th}\big)=0$ 
for every $f\in L$. Let $F^\perp$ and $\mu$
be the corresponding components of $F$
according to \eqref{representation T}. Thus, 
for each $n\in\N$, let $h_n\in\Ha$ be such 
that $F(\abs{f_n-Th_n})\le2^{-n}$. If $g\in\Ha$ 
and $Tg\le\abs{f_n-Th_n}$, then
\begin{align*}
\phi(g)
=
F(Tg)
\le
F(\abs{f_n-Th_n})
\le
2^{-n}.
\end{align*}
Thus $\sseqn{(h_n,f_n)}$ satisfies \eqref{ca} and,
by assumption, $\limsup_n\phi(h_n)\le0$. The 
inequality
$
\int f_nd\mu
	=
F(f_n)
	\le
\phi(h_n)+F(\abs{f_n-Th_n})
$
then proves that the functional $f\to\int fd\mu$ 
is a Daniel integral on the Stonean lattice 
$L\cap\B(\Omega)$ and it may thus be 
represented by some countably additive 
$(\Ring,\hat\mu)\in\Meas(\Omega)$. To prove 
that $\hat\mu$ agrees with $\mu$ over the whole 
of $L$ it is enough to remark that when $f\in L$ 
and $f\ge0$, then
$\hat\mu^*(f>k)
\le
k^{-1}\int f\wedge kd\hat\mu
=
k^{-1}\int f\wedge kd\mu
\le
k^{-1}F(f)$
and therefore $f\wedge k$ converges to $f$ in 
$L^1(\hat\mu)$.

\tiref{c}. 
For each $\alpha$ in a directed set $\mathfrak A$, 
let $F_\alpha\in\Lin{L}_+$ be such that 
$F_\alpha(Th)=\phi(h)$ for each $h\in\Ha$. Given
that $F_\alpha$ is conglomerative with respect to
the identity on $L$, it is of the form
\begin{equation}
F_\alpha(f)
	=
F_\alpha^\perp(f)+\int fd\mu_\alpha
\qquad
f\in L
\end{equation}
with $F^\perp_\alpha[L\cap\B(\Omega)]=\{0\}$ and
$(\Ring_\alpha,\mu_\alpha)\in\Meas(\Omega)$ such 
that $L\subset L^1(\mu_\alpha)$.
Observe that if $f\in L$ then there exists $h\in\Ha$
such that $\abs f\le Th$ and thus such that
$F_\alpha(\abs f)\le\phi(h)$. The net $\neta F$
admits then a subnet (still indexed by $\alpha$
for convenience) such that 
\begin{equation*}
F(f)
	=
\lim_\alpha F_\alpha(f)
\qquad
f\in L.
\end{equation*}
Since $F$ is positive we write it as
$F(f)=F^\perp(f)+\int fd\mu$. 
If the net $\neta\mu$ is increasing on 
$L_0\subset L\cap\B(\Omega)$ then
\begin{align*}
\lim_\alpha\int fd\mu_\alpha
	=
\lim_\alpha F_\alpha(f)
	=
F(f)
	=
\int fd\mu
\qquad
f\in\ L_0.
\end{align*}
It is clear that 
$\int fd\mu
	\ge
\int fd\mu_\alpha$ 
for each $\alpha\in\mathfrak A$ and $f\in L_0$.
By Zorn lemma this proves the existence of a
representing measure $\mu$ which is $L_0$-%
maximal.
\end{proof}

Before moving to applications we can generalize 
Theorem \ref{th representation} by dropping the 
assumption of linearity.

\begin{corollary}
\label{cor add}
Let $\Ha$ be a non empty set, let 
$T\in\bFun{\Ha,\Fun\Omega}$ 
be directed and denote by
$
L
$
the ideal generated by $T[\Ha]$. Then
$\phi\in\Fun{\Ha}$ is $T$-conglomerative
in the sense that
\begin{equation}
\label{conglo add}
\sum_{n=1}^Na_n\phi(h_n)<0
\qtext{implies}
\inf_\omega\sum_{n=1}^N
a_n(Th_n)(\omega)<0
\qquad
h_1,\ldots,h_N\in\Ha,
a_1,\ldots,a_N\in\R
\end{equation}
if and only if there exist 
(i)
$(\Ring,\mu)\in\Meas(\Omega)$ and 
(ii)
$F^\perp\in\Lin{L}_+$ 
such that
$F^\perp[L\cap\B(\Omega)]=0$,
\begin{equation}
L\subset L^1(\mu)
\qtext{and}
\phi(h)
=
F^\perp(Th)+\int Thd\mu
\qquad
h\in\Ha.
\end{equation}
Moreover, $\mu$ is a probability if and only if
\begin{equation}
\sum_{n=1}^Na_n\phi(h_n)
	\ge
\inf_\omega\sum_{n=1}^N
a_n(Th_n)(\omega)
\qquad h_1,\ldots,h_N\in\Ha,\ 
a_1\ldots,a_N\in\R.
\end{equation}
\end{corollary}

\begin{proof}
Let $e_h$ be the evaluation on $\Fun\Ha$
corresponding to $h\in\Ha$, that is 
$e_h(G)=G(h)$. If $V$ is a linear 
space, then each $G\in\Fun{\Ha,V}$  may be 
associated with a map $\hat G$ from the span
of $\{e_h:h\in\Ha\}$ into $V$ by letting
\begin{equation}
\label{hat}
\hat G\big(a_1e_{h_1}+\ldots+a_Ne_{h_N}\big)
	=
\sum_{n=1}^Na_nG(h_n)
\qquad
h_1,\ldots,h_N\in\Ha,\ 
a_1,\ldots,a_N\in\R.
\end{equation}
It is immediate that $\hat G$ is well defined 
and linear. Letting $\hat\phi$ and $\hat T$ be 
defined via \eqref{hat}, then \eqref{conglo add} 
is equivalent to the statement that $\hat\phi$ 
is $\hat T$-conglomerative while $\hat T$ is 
directed if and only if so is $T$. The claim 
follows from Theorem \ref{th representation}.
\end{proof}

A special case of Corollary \ref{cor add} applies
to the case in which $\neta\Ha$ is a family of
sets and, for each $\alpha\in\mathfrak A$,
$\phi_\alpha\in\Fun{\Ha_\alpha}$ and 
$T_\alpha\in\Fun{\Ha_\alpha,\Fun\Omega}$. 
Just let
$\Ha
	=
\{(h,\alpha):u\in\Ha_\alpha,\alpha\in\mathfrak A\}$,
$\phi(h,\alpha)=\phi_\alpha(h)$ and
$T(h,\alpha)=T_\alpha(h)$.

As pointed out by Choquet \cite[p. 325]{choquet 
positive I}, not all linear functionals admit an 
integral representation, not even finitely additive. 
This occurs, e.g., when $\Ha$ consists of 
polynomials and $\phi$ associates to each 
$h\in\Ha$ the coefficient of its term of degree 
$n$, for some fixed $n\ge1$. With the aim of 
extending the classical 
Riesz-Markoff theorem, Choquet assumes that 
$\Omega$ is a compact topological space, $\Ha$ 
a positively generated linear space of extended 
real-valued, continuous functions on $\Omega$ 
and takes $T$ to be a quotient $T(h)=h/g$. This 
construction permits to characterize positive linear 
functionals on $\Ha$ as a summable family of 
submeasures \cite[theorem 42]{choquet positive II}.

Theorem \ref{th representation} bears a closer 
relation to another result of Choquet, the extremal 
representation theorem, that was originally proved 
in \cite{choquet} and later variously extended and 
reformulated (see, \cite{choquet83}, 
\cite{lukes et al} or \cite{phelps} for an overview 
of this literature). To see this connection clearly, 
fix 
$\Omega=\Psi\subset\Lin{\Ha}$ 
and define 
$T\in\Lin{\Ha,\Fun\Psi}$ 
by letting $Th(\psi)=\psi(h)$ i.e. as the map that
associates each $h\in\Ha$ with the (restriction 
to $\Psi$ of the) corresponding evaluation $e_h$ 
on $\Fun\Ha$. It is then easily seen that 
conglomerability may be nicely restated in 
geometric terms as the condition 
\begin{equation}
\label{Psi-conglomerative}
\phi\in\ccon[\Ha](\Psi),
\end{equation}
i.e. as $\phi$ being an element of the closed, conical 
hull of $\Psi$, the closure being in the $\Ha$ topology. 
Likewise, the inequality 
$\inf_\omega Th(\omega)\le\phi(h)$
for all $h\in\Ha$ is equivalent to the condition
$\phi\in\cco[\Ha](\Psi)$.

In the light of these remarks the following result 
becomes obvious.

\begin{corollary}
\label{cor choquet}
Let $\Ha$ be a vector lattice and 
$\Psi\subset\Lin{\Ha}_+$. Then,
$\phi\in\ccon[\Ha](\Psi)$ 
if and only if there exist 
(i)
$\phi^\perp\in\Lin{\Ha}$ with $\phi^\perp(h)\ge0$ 
when $\inf_\psi\psi(h)>-\infty$ and 
(ii)
$(\Ring,\mu)\in\Meas(\Psi)$
such that
\begin{equation}
\label{choquet}
\restr{e_h}{\Psi}\in L^1(\mu)
\qtext{and}
\phi(h)
=
\phi^\perp(h)+\int_\Psi\psi(h)d\mu
\qquad
h\in\Ha.
\end{equation}
Moreover, $\mu$ is a probability 
if and only if $\phi\in\cco[\Ha](\Psi)$.
\end{corollary}

The lattice structure of $\Ha$ guarantees that the 
map $T$ defined above is directed, as in ($\alpha$).

To compare this result with the classical extremal 
or barycentrical representation, we remark that 
the conical structure and the choice of the $\Ha$ 
topology make the conglomerability condition
\eqref{Psi-conglomerative} a very weak 
restriction not requiring compactness nor 
boundedness and not relying as a consequence 
on the existence of extreme points. The first to 
obtain a proof of Choquet theorem without 
assuming compactness was Edgar 
\cite[theorem p. 355]{edgar} 
who considered a bounded, closed, convex, 
separable subset of a Banach space possessing 
the Radon Nikodym property and constructed
his proof exploiting norm convergence of vector 
valued martingales. 

Another version of Choquet theorem is obtained
starting from condition ($\beta$) for directedness 
of $T$ and requires boundedness.

\begin{corollary}
\label{cor choquet V}
Let $\Ha\subset\Fun{S}$ be a vector subspace,
$\phi\in\Lin\Ha$ and let $V\subset S$ be 
$\Ha$-bounded, i.e. such that
$
\sup_{v\in V}\abs{h(v)}<\infty
$
for all $h\in\Ha$.
Then,
\begin{equation}
\label{major}
\phi(h)
\ge
\inf_{v\in V}h(v)
\qquad
h\in\Ha
\end{equation}
if and only if there exists a probability structure 
$(\Ring,\mu)$ on $V$ such that
\begin{equation}
\label{choquet V}
h\in L^1(\mu)
\qtext{and}
\phi(h)
=
\int_Vh(v)d\mu
\qquad
h\in\Ha.
\end{equation}
\end{corollary}

\begin{proof}
Consider the vector space $\Ha\times\R$ as acting 
on $S$ via $(h,r)(s)=h(s)+r$. In the notation of 
Theorem \ref{th representation}, let $\Omega=V$, 
$T(h,r)=\restr{(h,r)}V$ and $\hat\phi(h,r)=\phi(h)+r$. 
By ($\beta$), $T$ is directed as $T[\Ha\times\R]$ 
is a subset of $\B(\Omega)$ containing the 
constants. \eqref{major} is equivalent to 
$\hat\phi(h,r)
\ge
\inf_v(Th)(v)
$,
i.e. to the representation of $\hat\phi$ in the form 
\eqref{representation T} form some probability 
structure $(\Ring,\mu)$ and with $F^\perp=0$
as $\hat T[\Ha\times\R]\subset\B(V)$.
\eqref{choquet V} follows upon restricting to
elements of the form $(h,0)$. The converse
implication is obvious.
\end{proof}

A clear example in which \eqref{major} holds is 
the one in which $\Ha$ consists of affine functions 
and $\phi(h)=h(u)$ for some $u\in\cco[\Ha](V)$. 
We highlight that Corollary \ref{cor choquet V} 
does not require topological assumptions; as a
drawback, the characterization of the mapping 
$u\to\mu_u$ is rather difficult. In the case in 
which $\Ha$ is a Stonean sublattice, however, the 
minimality property is enough to imply that $u\in V$ 
if and only if $\mu_u$ is the point mass measure 
at $u$.

\section{Finitely Additive Companions.}
\label{sec representations}

In this section we return to the problem of the 
existence of companions.

\begin{theorem}
\label{th companion}
Let $(\A,m)\in\Meas(\Omega)$,
$X\in\Fun{\Omega,S}$ and $\Ha$ a Stonean 
vector sublattice of $\Fun{S}$. Let 
$X'\in\Fun{\Omega',S}$. There is equivalence 
between the condition
\begin{equation}
\label{conglo X}
\int h(X)dm<0
\quad\text{implies}\quad
\inf_{\omega'\in\Omega'} h\big(X'(\omega')\big)
<
0
\qquad
h\in\Ha,\ 
h(X)\in L^1(m)
\end{equation}
and the existence of a minimal 
$(\Ring,\mu)\in\Meas(\Omega')$ satisfying
\begin{equation}
\label{companion} 
h(X')\in L^1(\mu)
\quad\text{and}\quad
\int h(X)dm
	=
\int h(X')d\mu
\qquad
h\in\Ha,\ 
h(X)\in L^1(m).
\end{equation}
In addition, 
\begin{enumerate}[(a)]
\item
$\mu$ is a probability if and only if
\begin{equation}
\label{conglo X prob}
\int h(X)dm
	\ge
\inf_{\omega'\in\Omega'} h\big(X'(\omega')\big)
\qquad
h\in\Ha,\ 
h(X)\in L^1(m);
\end{equation}
\item
if $X'[\Omega']$ is closed in the topological 
space $S$ and $\Ha\subset\Cts{S}$ then $\mu$ 
is countably additive if either
(i)
$\Ha\subset\Cts[K]{S}$,
(ii)
$X'$ is $\mu$-tight
or
(iii)
$X$ is $m$-tight and
$
m_*\big(X\notin X'[\Omega']\big)=0
$.
\end{enumerate}
\end{theorem}

\begin{proof}
\eqref{conglo X} is equivalent to $\phi$ 
being $T$-conglomerative with 
$\phi(h)=\int h(X)dm$ 
and $Th=h(X')$ for every $h\in\Ha$. Thus, 
\eqref{companion} follows from 
\eqref{representation T} after noting that, 
in the present setting, 
$\phi(h)=\lim_k\phi(h\wedge k)$ for every 
$h\in\Ha_+$. That \eqref{conglo X prob} is
necessary and sufficient for $\mu$ to be a
probability follows directly from Theorem 
\ref{th representation}.\tiref{a}.

Let $X'[\Omega']$ be closed and $\seqn h$ a 
sequence in $\Ha\subset\Cts S$ with $h_n(X')$ 
decreasing to $0$, i.e. $h_n$ decreasing to $0$ 
on $X'[\Omega']$. We claim that \tiref{i}, \tiref{ii} 
or \tiref{iii} imply $\lim_n\int h_n(X')d\mu=0$. If 
$\Ha\subset\Cts[K]S$, then in computing such 
limit one may replace $S$ with some compact 
subset so that \tiref{i} follows from \tiref{ii}. Fix 
$\varepsilon>0$. Under \tiref{ii} there exists 
$K'\subset S$ compact and $B'^c\in\Ring$ such 
that $B'\subset\{X'\in K'\}$ and
\begin{align*}
\int h_n(X')d\mu
\le
\int h_n(X')\set{B'}d\mu+\varepsilon
\qquad
n\in\N.
\end{align*}
But then 
$\lim_n\sup_{\omega'\in B'}h_n(X')
\le
\lim_n\sup_{s\in X'[\Omega']\cap K'}h_n(s)
=
0$, 
by Dini's theorem. Under \tiref{iii}, we can find
an extension $\bar m$ of $m$ to the minimal
ring containing the set $F=\{X\notin X'[\Omega']\}$ 
such that $\bar m(F)=0$. We can also find
$K\subset S$ compact and $B^c\in\A$ such
that $B\subset\{X\in K\}$ and that
\begin{align*}
\int h_n(X')d\mu
=
\int h_n(X)dm
=
\int h_n(X)d\bar m
\le
\int h_n(X)\set{B\setminus F}d\bar m
+
\varepsilon
\qquad
n\in\N
\end{align*}
so that again 
$
\lim_n\sup_{\omega\in B\setminus F}h_n(X)
\le
\lim_n\sup_{s\in X'[\Omega']\cap K}h_n(s)
=
0
$.
In either case the positive linear functional 
$\int h(X')d\mu$ on the Stonean lattice 
$\Ha[X']$ is a Daniell integral and it may be 
represented via a countably additive set 
function. Since $\mu$ is minimal, it must then 
be countably additive.
\end{proof}

To clarify the connection with Doob's work,
consider a $\pi$-strategy, i.e. a function
$\sigma(h\vert B)$ where $h$ runs across the
family $\B(\Omega)$ of bounded functions 
on $\Omega$ and $B$ is an element of the 
partition $\pi$ of $\Omega$. As in other papers
on finitely additive probability (see e.g. Regazzini 
\cite{regazzini}) conditional expectation is 
defined setwise rather than as a measurable
function, as in Kolmogorov classical construction.
One notices that $m$ is $\sigma$-conglomerative 
in the sense of \cite[p. 90]{dubins}  if and only if 
\eqref{conglo X} holds with $\Ha=\B(\Omega)$,
$S=\Omega$, $X$ the identity map and 
$h(X')
=
\sum_{B\in\pi}\sigma(h\vert B)\set B$.

In the absence of restrictions on $\mu$, the 
existence of companions is guaranteed under 
a weak condition such as \eqref{conglo X}, 
namely if $X$ is $X'$-conglomerative. An 
obvious companion to any $X$ is the identity 
map on $\Omega'=S$. Given that being companion 
(relatively to the one given family $\Ha$) is a 
transitive property, the problem in Theorem 
\ref{th companion} may be simplified with no 
loss of generality by assuming that $X$ is the 
identity map on $\Omega=S$. In this case, if 
$m$ consists of sample frequencies, then the 
condition $m^*(X'[\Omega'])=0$ sufficient for 
$m$ to be $X'$-conglomerative means that all 
the observations in the given sample must belong 
to the range of $X'$. 

The existence of a countably additive companion 
was proved under \tiref{ii} by Dubins and Savage 
\cite[p. 190]{dubins savage}, for the case 
$\Omega=\Omega'=S=\R$, and has then been 
revived and extended to the case $S=\R^n$ by 
Karandikar, \cite{karandikar} and 
\cite{karandikar jmva}, who used it in the proof 
of finitely additive limit theorems. The conditions 
for the existence of a countably additive companion 
obtained in Theorem \ref{th companion} may be 
employed to refine the results of the preceding 
section. In particular if the set $\Psi$ in Corollary 
\ref{cor choquet} is $\Ha$-compact then in 
\eqref{choquet} one has $\phi^\perp=0$ and 
$\mu$ can be chosen to be countably additive.

An interesting issue concerns the construction of 
an auxiliary state space on which every function
$X$ admits a countably additive companion. 

\begin{lemma}
\label{lemma cc}
Let $(\A,m)\in\Meas(\Omega)$,
$S$ be a metric space, $s_0\in S$, 
$X\in\Fun{\Omega,S}$ and 
$\tilde\Omega=\Fun{\N,\Omega}$. 
Define $\tilde X\in\Fun{\tilde\Omega}$
as
\begin{equation}
\label{tilde X}
\tilde X(\tilde\omega)
	=
\lim_kX(\omega_k)
\qtext{if the limit exists or else}
\tilde X(\tilde\omega)=s_0,
\qquad
\tilde\omega=\seq\omega k\in\tilde\Omega.
\end{equation}
There exists $(\Ring,\mu)\in\Meas(\tilde\Omega)$  
countably additive and such that $(\tilde X,\mu)$ 
is companion to $(X,m)$ relatively to $\Cts[K]S$. 
Moreover, if $S=\Fun\N$ and $X_n$ is $m$-convergent 
(resp. converges in $L^1(m)$)  to $0$ then 
$\tilde X_n$ is $\mu$-convergent (resp. 
converges in $L^1(\mu)$) to $0$.
\end{lemma}

\begin{proof}
$X$ is $\tilde X$-conglomerative relatively to 
any $\Ha\subset\Fun S$ since
$X[\Omega]\subset\tilde X[\tilde\Omega]$;
moreover, $\tilde X[\tilde\Omega]$ is closed. 
The first claim follows from Theorem 
\ref{th companion}.\tiref{b}.

Let $S=\Fun\N$ and replace $m$ with some 
positive extension $\bar m$ to the ring 
$
\{A\subset\Omega:m^*(A)<\infty\}
$.
By the first claim there exists 
$(\Ring,\mu)\in\Meas(\tilde\Omega)$
countably additive such that $(X,\bar m)$ 
and $(\tilde X,\mu)$ 
are companions relatively to $\Cts[K]{S}$ 
-- and {\it a fortiori} so are $(X,m)$ and
$(\tilde X,\mu)$. Fix $b>a>0$ and $k>0$
and let $g,f_k\in\Cts\R$ be such that
$\sset{x>b}<g(x)\le\sset{x>a}$
and $\sset{x<k-1}<f_k(x)<\sset{x<k}$
so that $f_k\uparrow1$. Writing
$h_n(X)=g(\abs{X_n})$ and 
$h_n^k(X)=h_n(X)f_k(\abs{X_n})$,
$h_n\in\Cts{S}$ and 
$h_n^k\in\Cts[K]{S}$. But then,
\begin{align*}
m^*(\abs{X_n}>a)
	\ge
\int h_n(X)d\bar m
	\ge
\lim_k\int h_n^k(X)d\bar m
	=
\lim_k\int h_n^k(\tilde X)d\mu
	=
\int h_n(\tilde X)d\mu
	\ge
\mu^*(\abs{\tilde X_n}>b)
\end{align*}
and, consequently,
\begin{align*}
\int_\pm^\infty m^*(\abs{X_n}>t)dt
	\ge
\int_b^\infty \mu^*(\abs{\tilde X_n}>t)dt
	\ge
\int_b^\infty \mu_*(\abs{\tilde X_n}>t)dt
\end{align*}
so that
$\int\abs{X_n}dm\ge\int\abs{\tilde X_n}d\mu$
whenever $X_n\in L^1(m)$, by Lemma 
\ref{lemma L1}. 
\end{proof}

Lemma \ref{lemma cc} may help understanding 
the connection between convergence pointwise 
and in measure under finite additivity, i.e. when 
Egoroff theorem fails. We establish that a condition 
weaker than uniform  convergence may be assumed.

\begin{corollary}
\label{cor pointwise}
Let $(\Omega,\A,m)$ be a probability space 
and $\seqn X$ a $m$-measurable sequence 
in $\Fun\Omega$. Assume that 
\begin{align}
\label{limlim}
\lim_n\lim_kX_n(\omega_k)
=
0,
\end{align}
whenever $\lim_kX_n(\omega_k)$ exists for 
all $n\in\N$. Then, $X_n$ $m$-converges 
to $0$.
\end{corollary}

\begin{proof}
Write $Y=\sseqn{\abs{X_n}\wedge1}$ and 
define $\tilde\Omega$ and $\tilde Y$ as in 
\eqref{tilde X}, with $s_0=0$. By Lemma 
\ref{lemma cc} there exists a countably
additive $(\Ring,\mu)\in\Meas(\tilde\Omega)$ 
such that $(Y,m)$ and $(\tilde Y,\mu)$ are 
companions relatively to $\Cts[K]{\Fun\N}$. 
Fix $\tilde\omega=\seq\omega k$ in 
$\tilde\Omega$. If $Y$ does not converge 
along $\tilde\omega$ then 
$Y_n(\tilde\omega)=0$, otherwise
$
\lim_n\tilde Y_n(\tilde\omega)
=
\lim_n\lim_k Y_n(\omega_k)
=
0
$,
by \eqref{limlim}. But then countable
additivity implies
$
0
	=
\lim_n\int \tilde Y_nd\mu
	=
\lim_n\int Y_nd m
$
so that $X_n$ $m$-converges to $0$.
\end{proof}

In Theorem \ref{th companion}  the set function 
$\mu$ is completely unrestricted. A possible mitigation 
is to require that $\mu$ vanishes on some suitable, 
given collection $\Neg$ of subsets of $\Omega$.

\begin{theorem}
\label{th representation N}
In the same setting as Theorem \ref{th companion}, 
let $\Neg$ an ideal of subsets of $\Omega'$. The 
condition
\begin{equation}
\label{cond N}
\int h(X)dm<0
\quad\text{implies}\quad
\sup_{N\in\Neg}\inf_{\omega'\in N^c} h
\big(X'(\omega')\big)<0
\qquad
h\in\Ha
\end{equation}
is equivalent to the existence of a minimal 
$(\Ring,\mu)\in\Meas(\Omega')$ 
which satisfies $\Neg\subset\Ring$, 
\begin{equation}
\label{Neg representation}
\mu[\Neg]=\{0\},
\quad
h(X')\in L^1(\mu)
\quad\text{and}\quad
\int h(X)dm
	=
\int h(X')d\mu
\qquad
h\in\Ha.
\end{equation}
Moreover, 
(a)
$\mu$ is a probability if and only if
\begin{equation}
\int h(X)dm
	\ge
\sup_{N\in\Neg}\inf_{\omega'\in N^c} h
\big(X'(\omega')\big)
\qquad
h\in\Ha,
\end{equation}
(b)
if $\A$ is a $\sigma$ ring, $m$ is countably 
additive and $\Neg$ a $\sigma$ ideal then 
$\mu$ is countably additive provided
$m_*\big(X\notin X'[N^c]\big)=0$ for all 
$N\in\Neg$.
\end{theorem}

\begin{proof}
Since $\Neg$ is an ideal,  the binary relation $\succeq$ 
on $\Fun{\Omega'}$ defined by letting
\begin{equation}
\label{pos}
f\succeq g
\quad\text{if and only if}\quad
\sup_{N\in\Neg}\inf_{\omega'\in N^c}(f-g)(\omega')
	\ge
0
\qquad
f,g\in\Fun{\Omega'}
\end{equation}
is a partial order and $f\ge g$ implies $f\succeq g$. 
Moreover, $f_i\succeq g_i$ for $i=1,2$ implies 
$f_1\vee f_2\succeq g_1\vee g_2$. In fact, 
$f_1\vee f_2\succeq f_i\succeq g_i$ i.e. 
$f_1\vee f_2\ge g_i-\varepsilon$ outside of some 
$N_i\in\Neg$.
Thus, $f_1\vee f_2\ge g_1\vee g_2-\varepsilon$ 
outside of $N_1\cup N_2\in\Neg$ which, by 
\eqref{pos}, is equivalent to 
$f_1\vee f_2
	\succeq 
g_1\vee g_2$. 
It is easy to see that, 
relatively to pointwise ordering, the set
\begin{equation}
\F
	=
\big\{f\in\Fun{\Omega'}:f\sim h(X')
\text{ for some }h\in\Ha\big\}
\end{equation}
is a Stonean vector sublattice of $\Fun{\Omega'}$. 
Writing
\begin{equation}
\label{phi}
\phi(f)
	=
\int h(X)dm
\qquad
f\sim h(X'),\ 
h\in\Ha
\end{equation}
implicitly defines, via \eqref{cond N}, a positive 
linear functional on $\F$ so that, by Corollary 
\ref{cor choquet}, we conclude that there exists 
a minimal measurable structure $(\Ring,\mu)$ 
on $\Omega'$ satisfying
\begin{equation}
f\in L^1(\mu)
\quad\text{and}\quad
\phi(f)=\int fd\mu
\qquad
f\in\F.
\end{equation}
Observe that if $N\in\Neg$ then $\set N\sim0$:
thus, $\set N\in\F$, $N\in\Ring$ and $\mu(N)=0$. 
This proves \eqref{Neg representation} while 
the converse implication, is obvious. The proof 
of claim \tiref{a} is easily obtained from the 
one of the corresponding claim in Theorem 
\ref{th representation}. Eventually we prove 
\tiref{b}, once again, by showing that under the 
stated conditions the functional $\phi$ defined 
in \eqref{phi} is a Daniell integral over $\F$. In 
fact, let $\seqn f$ be sequence in $\F$ decreasing 
pointwise to $0$ with $f_n\sim h_n(X')$ and 
$h_n\in\Ha$, $n=1,2,\ldots$. Define 
$g_n=\bigwedge_{1\le j\le n}h_j$ and 
$g=\lim_ng_n$. As shown above, 
$f_n\sim g_n(X')\succeq g(X')$ so that, 
by the assumption that $\Neg$ is a $\sigma$ ideal,
$
\{g(X')>\varepsilon\}
	\subset
\bigcup_n\{g(X')\ge f_n+\varepsilon\}
	\in
\Neg
$
and 
$\{g>\varepsilon\}
	\subset 
X'[\{g(X')\le\varepsilon\}]^c$. Given that $\A$
is a $\sigma$ ring, we conclude that
$m(g(X)>\varepsilon)=0$ and so 
$
\lim_n\phi(f_n)
	=
\lim_n\int g_n(X')d\mu
	=
\lim_n\int g_n(X)dm
	=
\int g(X)d m
	=
0
$.
\end{proof}

\begin{example}\label{ex normal}
Let $(\Omega',\A,P)$ be a classical probability space, 
$S=\R$ and let $X'$ be a normally distributed random 
quantity on $\Omega'$. Fix $m\in fa(\Bor(\R))_+$ 
arbitrarily and let $\Ha=\Cts\R\cap L^1(m)$. Given 
that $P(X'\in B)>0$ for every $B$ open, we conclude 
that $m$ is $X'$-conglomerative relatively to $\Ha$. 
In other words a normally distributed random quantity 
can assume any arbitrary distribution (relatively to 
the continuous functions) upon an accurate choice 
of the reference measure.

In addition, let $\Neg$ consist of all $P$ null sets 
and observe that $X'[N^c]^c$ has $0$ Lebesgue 
measure -- as $P(X'\in X'[N^c]^c)=P(N)=0$ and 
the $P$ distribution of $X'$ is mutually absolutely
continuous with respect to Lebesgue measure -- 
and has therefore empty interior -- so that 
$\cl{X'[N^c]}=\R$. Therefore,
\begin{align*}
\sup_{N\in\Neg}\inf_{\omega\in N^c} h(X'(\omega))
	=
\sup_{N\in\Neg}\inf_{s\in X'[N^c]} h(s)
	=
\sup_{N\in\Neg}\inf_{s\in \cl{X'[N^c]}} h(s)
	=
\inf_{s\in\R}h(s)
\qquad
h\in\Cts\R.
\end{align*}
Property \eqref{cond N} then holds for every 
$m\in fa(\Bor(\R))_+$ with $\Ha=\Cts\R$. One
may then find $\mu$ vanishing on $\Neg$ and
such that $(X',\mu)$ is companion to $m$.

Even if $m$ were countably additive, $\mu$ need 
not be so. The Dirac measure is a good case in 
point of a regular, countably additive measure 
that cannot be represented as the distribution of 
$X'$ with respect to some countably additive 
representing measure $\mu$ which vanishes
on $P$ null sets. To this end we may assume 
in addition that $m$ does not charge sets with 
empty interior. Under this further assumption, 
$m_*(X'[N^c]^c)=0$ so that $\mu$ is countably 
additive by virtue of Theorem 
\ref{th  representation N}.\tiref{b} and vanishes
on $N\in\Neg$. Of course the same conclusion 
holds upon replacing $X'$ with any variable 
possessing a strictly positive density over the 
whole of $\R$. When $m$ and $\mu$ are 
countably additive, one may exploit the fact 
that the indicator of each open subset $B$ of 
$\R$ is the pointwise limit of an increasing 
sequence $\seqn f$ of continuous functions, 
and conclude
\begin{align}
\label{distr}
\mu(X'\in B)
	=
\lim_n\int f_n(X')d\mu
	=
\lim_n\int f_ndm
	=
m(B).
\end{align}
\end{example}

The preceding example may be generalized into
the following:

\begin{theorem}
\label{th representation norm}
Let $\Ha\subset\Cts\R$ be a Stonean sublattice,
$\phi\in\Lin{\Ha}_+$, $X'$ a normally 
distributed random quantity on a standard 
probability space $(\Omega',\A,P)$ and $\Neg$
the collection of all $P$ null sets. There exists 
a minimal $(\Ring,\mu)\in\Meas(\Omega')$ such 
that $\Neg\subset\Ring$, $\mu$ vanishes on 
$\Neg$ and
\begin{equation}
\label{norm representation}
\phi(h)
	=
\int h(X')d\mu
\qquad
h\in\Ha.
\end{equation}
Moreover, if $\Ha$ is an ideal in $\Cts\R$ then
\begin{enumerate}[(i)]
\item
$\mu$ is countably additive if and only if 
$\lim_n\phi(h_n)=0$ for any decreasing 
sequence  $\seqn h$ in $\Ha_+$ which 
converges to $0$ in Lebesgue measure,
\item

$\mu$ is countably additive and $\mu^*(X'\in C)=0$
when $C$ has empty interior if and only if 
$\lim_n\phi(h_n)=0$ for any decreasing sequence 
$\seqn h$ in $\Ha_+$ admitting $0$ as the
largest continuous function dominated by
$\inf_nh_n$.
\end{enumerate}
\end{theorem}

\begin{proof}
A positive linear functional on a vector lattice 
is conglomerative with respect to the identity, 
in its turn a directed map. The representation
of $\phi$ as $\int hdm$, with $m$ minimal, 
follows from Theorem \ref{th representation}; 
\eqref{norm representation} from Example 
\ref{ex normal}. If $\Ha$ is an ideal and $\phi$ 
meets either property, \tiref{i} or \tiref{ii}, then 
it is a Daniell integral and $m$ is a countably 
additive, regular measure on the generated 
$\sigma$ ring, still denoted by $\Ring$. We 
also notice that the indicator of a closed set 
$F\in\Ring$ may be expressed as the pointwise 
limit of a decreasing sequence $\seqn h$ of 
positive, continuous functions with $0\le h_n\le 1$. 
Fix $h\in\Ha_+$. Since $\Ha$ is an ideal, 
$hh_n\in\Ha$ for each $h\in\Ha_+$ and thus
$
\int h\set Fdm
	=
\lim_n\int hh_ndm
	=
\lim_n\phi(hh_n)
$.
Then $\int h\set Fdm=0$ in two different 
situations: when $F$ has $0$ Lebesgue measure 
and $\phi$ satisfies \tiref{i} (as $hh_n$ converges 
then to $0$ in Lebesgue measure) or if $F$ is 
nowhere dense and $\phi$ satisfies \tiref{ii} (as 
$0$ is then the largest, continuous function 
dominated by $h\set F$). In either case the 
restriction of $m$ to $F^c$ is another representing 
measure for $\phi$ so that, by minimality, $m(F)=0$. 
Given that $X'[N^c]^c$ has $0$ Lebesgue measure 
and empty interior when $N\in\Neg$ and that $m$ 
is regular, then \tiref{i} and \tiref{ii} imply 
$m_*(X'[N^c]^c)=0$ and, by Theorem 
\ref{th representation N}, that $\mu$ is countably 
additive. Assume, conversely, that $\mu$ is 
countably additive and let $\seqn h$ be a 
decreasing sequence in $\Ha_+$ with pointwise 
limit $h$. For each fixed $\varepsilon>0$ we
obtain that $\mu^*(h(X')>\varepsilon)=0$ in the
following two cases: when $h_n$ decreases to $0$ 
in Lebesgue measure and $\mu$ meets \tiref{i} (as 
the set $\{h>\varepsilon\}$ has $0$ Lebesgue measure
and thus $\{h(X')>\varepsilon\}\in\Neg$) or when
$0$ is the largest, continuous function dominated 
by $h$ and $\mu$ meets \tiref{ii} (as 
$\{h>\varepsilon\}$ has then empty interior). In 
either case 
$\lim_n\phi(h_n)
=
\lim_n\int h_n(X')d\mu
=
\int h(X')d\mu
=
0$.
\end{proof}

It is implicit in Theorem \ref{th representation norm}
that a normally distributed random quantity may
assume whatever distribution upon a change of the
reference measure and whatever distribution
absolutely continuous with respect to Lebesgue 
measure upon an absolutely continuous change
of the original probability $P$. A version of this 
result will be established with Brownian motion
replacing normal random quantities.

We now show that the existence of companions
may be obtained even outside of the linear case.
Eventually, we turn attention to convex functions. For 
$f\in\Fun\R$ we denote by $D^+f$ and $D^-f$ the 
right and left derivatives and by $f(x+)$ and $f(x-)$ 
the right and left limits at $x$, provided such 
quantities exist. We also set conventionally 
\begin{align*}
D^+f(\infty)=D^-f(\infty)=\lim_{x\to\infty}D^+f(x)
\quad\text{and}\quad
D^+f(-\infty)=D^-f(-\infty)=\lim_{x\to-\infty}D^+f(x).
\end{align*} 
Observe that if $x_0\in\arginf_{x\in\RR}f(x)$, then 
$D^+f(x_0),D^-f(x_0)\in\R$ and that for this reason, 
upon replacing $f$ with the function
$\hat f(x)
=
f(x)
	-
[D^+f(x_0)\sset{x>x_0}+D^-f(x_0)\sset{x\le x_0}]$
we may assume $D^+f(x_0)=D^-f(x_0)=0$.

\begin{theorem}
\label{th convex}
Let $\varphi\in\Fun{\R}$, 
$x_0\in\arginf_{x\in\RR}\varphi(x)$
and assume  
$D^+\varphi(x_0)=D^-\varphi(x_0)=0$.
Define
\begin{equation}
h_u^v(x)
	=
(v-x\vee u)^+\sset{x>x_0}
-
(v\wedge x-u)^+\sset{x\le x_0}
\qquad
x,u,v\in\R.
\end{equation}
Let $\Neg$ be an ideal of subsets of $\Omega$ and 
$X\in\Fun\Omega$.
The following properties are mutually equivalent:
\begin{enumerate}[(i)]
\item\label{iphi}
$\varphi$ is convex and 
$\{u<X<v\}\in\Neg$ implies 
$D^-\varphi(v)\le D^+\varphi(u)$;
\item\label{ilambda}
there exists a $(\Ring,\lambda)\in\Meas(\Omega)$ 
such that 
(a) 
$\Neg\subset\Ring$ and $\lambda[\Neg]=\{0\}$,
(b)
$\lim_n\lambda^*(\abs{X-x_0}<2^{-n})=0$, 
(c)
$\{h_u^v(X):v,u\in\R\}\subset L^1(\lambda)$ and
\begin{equation}\label{riesz lambda}
\varphi(v)
	=
\varphi(u)
	+
\int h_u^v(X) d\lambda
\qquad
v\ge u;
\end{equation}

\item\label{inu}
there exists 
$\nu\in fa(\Bor(\R))_+$ 
countably additive such that
(a)
$\nu(A)=0$ for $A$ open and $X^{-1}(A)\in\Neg$, 
(b)
$\nu^*(\{x_0\})=0$,
(c)
$\{h_u^v:v,u\in\R\}\subset L^1(\nu)$ and
\begin{equation}\label{riesz nu}
\varphi(v)
	=
\varphi(u)
	+
\int h_u^vd\nu
\qquad
v\ge u.
\end{equation}
\end{enumerate}
\end{theorem}

\begin{proof}
\imply{iphi}{ilambda}.
Write 
$\mathcal D
=
\big\{t:D^-\varphi(t)=D^+\varphi(t)\big\}\cup\{x_0\}$
and define 
$A_u=\{u<X\le x_0\}$, 
$A^v=\{x_0< X\le v\}$ 
and
\begin{equation}
\Ring_0
	=
\Big\{\big(A_u\cap N^c_u\big)
\cup
\big(A^v\cap N^c_v\big)\cup N:
u,v\in\mathcal D,\ 
N_u,N_v,N\in\Neg\Big\}.
\end{equation}
It is clear that $\Ring_0$ contains $\Neg$ (upon 
taking $u=v=x_0$) as well as 
$\{A_u,A^v:u,v\in\mathcal D\}$. 
Moreover, it is routine to verify that $\Ring_0$ is 
closed with respect to union and intersection with
\begin{subequations}
\label{lat}
\begin{equation}
H_1\cup H_2
=
\big(A_{u_1\wedge u_2}\cap N^c_u\big)
	\cup
\big(A^{v_1\vee v_2}\cap N^c_v\big)
	\cup 
N
\end{equation}
\begin{equation}
H_1\cap H_2
=
\big(A_{u_1\vee u_2}\cap\hat N^c_u)
	\cup
\big(A^{v_1\wedge v_2}\cap\hat N^c_v\big)
	\cup
\hat N
\end{equation}
\end{subequations} 
whenever 
$H_i
=
\big(A_{u_i}\cap N^c_{u_i}\big)
	\cup
\big(A^{v_i}\cap N^c_{v_i}\big)
	\cup
N_i
	\in
\Ring_0$
for $i=1,2$.
Write 
$
F(x)
	=
D^+\varphi(x\vee x_0)
+
D^-\varphi(x\wedge x_0)
$
and
\begin{equation}
\label{lambda0}
\lambda_0(H)
	=
F(v\vee x_0)-F(u\wedge x_0)
\quad\text{when}\quad
H
=
(A_u\cap N_u^c)\cup(A^v\cap N_v^c)\cup N
\in
\Ring_0.
\end{equation}
To see that $\lambda_0$ is well defined observe 
that if $u_1\wedge x_0<u_2\wedge x_0$ and
\begin{align*}
\big(A_{u_1}\cap N_{u_1}^c\big)
\cup
\big(A^{v_1}\cap N^c_{v_1}\big)\cup N_1
	=
\big(A_{u_2}\cap N_{u_2}^c\big)
\cup
\big(A^{v_2}\cap N^c_{v_2}\big)\cup N_2
	\in
\Ring_0
\end{align*}
then 
$
\{u_1\wedge x_0<X\le u_2\wedge x_0\}
	\in
\Neg
$.
Thus by \tiref{i} and the fact that 
$u_1,u_2\in\mathcal D$
and that $u_1<x_0$, 
\begin{align*}
D^-\varphi(u_1\wedge x_0)
	=
D^-\varphi(u_2\wedge x_0)
\qtext{i.e.}
F(u_1\wedge x_0)=F(u_2\wedge x_0)
\end{align*}
and likewise
$F(v_1\vee x_0)=F(v_2\vee x_0)$. In other words 
$\lambda_0\in fa(\Ring_0)_+$ with 
$\lambda[\Neg]=\{0\}$. 
Moreover, if $H_1,H_2\in\Ring_0$ then by \eqref{lat}
\begin{align*}
\lambda_0(H_1)+\lambda_0(H_2)
	&=
F(v_1\vee x_0)+F(v_2\vee x_0)-F(u_1\wedge x_0)-
F(u_2\wedge x_0)
	\\&=
F(v_1\vee v_2\vee x_0)
	+
F((v_1\wedge v_2)\vee x_0)
	-
F((u_1\vee u_2)\wedge x_0)
	-
F(u_1\wedge u_2\wedge x_0)
	\\&=
\lambda_0(H_1\cup H_2)+\lambda_0(H_1\cap H_2)
\end{align*}
i.e. $\lambda_0$ is strongly additive on $\Ring_0$. 
It follows from \cite[3.1.6 and 3.2.4]{rao} that 
$\lambda_0$ admits a unique extension 
$\lambda_1\in fa(\Ring_1)_+$ to the generated 
ring $\Ring_1$. Let $I$ be an interval with 
endpoints in $\R\cup\{x_0\}$. Given that 
$\mathcal D$ is dense in $\R\cup\{x_0\}$, 
$\lambda^*(X\in I)<\infty$. By 
\cite[3.4.1 and 3.4.4]{rao} we obtain a further 
extension $\lambda\in fa(\Ring)_+$ to the ring
$
\Ring
	=
\big\{A\subset\Omega:\lambda_1^*(A)<\infty\big\}
$.
Then $\{X\in I\}\in\Ring$ and $X\set I(X)$ is 
$\lambda$-measurable whenever $I$ is as above, 
by Lemma \ref{lemma measurable}. Therefore, 
\begin{align*}
\int_{u\vee x_0}^{v\vee x_0}D^+\varphi(t)dt
	&=
\int_{u\vee x_0}^{v\vee x_0}\set{\mathcal D}
[D^+\varphi(t)-y_0^+]dt\\
	&=
\int_u^v\set{\mathcal D}\lambda_1(x_0<X\le t)dt\\
	&=
\int_u^v\lambda(x_0<X\le t)dt\\
	&=
\int_{x_0}^\infty(v-u\vee X)^+d\lambda
&\text{(by Lemma \ref{lemma L1})}
\end{align*}
and similarly 
$
\int_{u\wedge x_0}^{v\wedge x_0}D^+\varphi(t)dt
	=
	-
\int_{-\infty}^{x_0}(v\wedge X-u)^+d\lambda
$.
We conclude 
\begin{align*}
\varphi(v)-\varphi(u)
	&=
\int_{u\vee x_0}^{v\vee x_0}D^+\varphi(t)dt
+
\int_{u\wedge x_0}^{v\wedge x_0}D^-\varphi(t)dt
	=
\int h_u^v(X)d\lambda.
\end{align*}
Fix an increasing $\seqn u$ and a decreasing 
$\seqn v$ sequence in $\mathcal D$ converging 
to $x_0$, with $u_n<u_{n+1}<x_0$ if 
$x_0>-\infty$ and $v_n>v_{n+1}>x_0$ if 
$x_0<\infty$. Then,
\begin{align*}
\lim_n\lambda^*(u_n<X<v_n)
	&\le
\lim_nD^+\varphi(v_n)
  -
D^-\varphi(u_n)
	=
0
\end{align*}
so that $\lim_n\lambda^*(\abs{X-x_0}<2^{-n})=0$.

\timply{ii}{iii}. 
With $u_n$ and $v_n$ defined as above, define the
function
\begin{align*}
h_u^v(x;n)=
\begin{cases}
h_u^v(x),&
\text{if }
x\notin(u_n,v_n]\\
h_u^v(u_n)\frac{u_{n+1}-x}{u_{n+1}-u_n},&
\text{if }
x\in (u_n,u_{n+1}]\\
h_u^v(v_n)\frac{x-v_{n+1}}{v_n-v_{n+1}},&
\text{if }
x\in(v_{n+1},v_n].
\end{cases}
\end{align*}
Then, $h_u^v(\cdot;n)$ is a continuous function 
vanishing outside of the interval 
$[u\wedge v_{n+1},v\vee u_{n+1}]$.
Moreover: 
\tiref{a}
$\big\{\babs{h_u^v(x;n)-h_u^v(x)}>c\big\}
\subset
(u_n,v_n]$
so that $h_u^v(X;n)$ is $\lambda$-convergent to 
$h_u^v(X)$,
\tiref{b}
$\babs{h_u^v(x;n)}\le\babs{h_u^v(x;n+1)}
\le
\babs{h_u^v(x)}$ 
,
\tiref{c}
$\lim_nh_u^v(x;n)=h_u^v(x)$ for all $x\ne x_0$ 
and 
\tiref{d}
$h_u^v(X;n)$ is $\lambda$-measurable and 
therefore an element of $L^1(\lambda)$. Let 
$(X',\nu)$, with $\Omega'=\R$ and $X'$ the 
identity, be the countably additive companion 
of $(X,\lambda)$ relatively to the family 
$\{h(X):h\in\Cts[K]\R\}$. It follows that
\begin{align*}
\int h_u^v(X)d\lambda
=
\lim_n\int h_u^v(X;n)d\lambda
=
\lim_n\int h_u^v(x;n)d\nu
=
\int h_u^vd\nu.
\end{align*}
Observe that if $x_0\in\R$ and 
$g_n\in\Cts[K]\R$ is such that 
$\set{(u_n,v_n]}
	\ge 
g_n\ge\set{(u_{n+1},v_{n+1}]}$,
then
\begin{align*}
\nu^*(\{x_0\})
\le
\lim_n\int g_n(X)d\lambda
\le
\lim_n\lambda(u_n<X\le v_n)
=
0.
\end{align*}

Let $I\subset\R$ be an open interval with 
$X^{-1}(I)\in\Neg$ and $\seqn g$ a sequence 
of non negative, continuous functions which 
increases to $\set I$. It is then obvious that
\begin{align*}
0
	=
\lim_n\int g_n(X)d\lambda
	=
\lim_n\int g_nd\nu
	=
\nu(I).
\end{align*}
The conclusion extends to open sets.

\imply{inu}{iphi}.
If $\varphi$ satisfies \eqref{riesz nu} it is clearly 
convex since the function $v\to h_u^v(x)$ is 
convex for every $u\le v$. Assume that $u<v$ 
and $\{u<X<v\}\in\Neg$. Then, $\nu((u,v))=0$ 
so that, for arbitrary $u<t<v$
\begin{equation}
\frac{\varphi(v)-\varphi(u)}{v-u}
	=
\begin{cases}
\nu([x_0,t)),&
	\text{if  }
v>u\ge x_0\\
\nu([t,x_0)),&
	\text{if }
x_0\ge v>u\\
0,&
	\text{if }
v> x_0> u
\end{cases}
\end{equation}
and \tiref{i} follows.
\end{proof}

If, e.g., $\varphi$ is differentiable at $x_0$,
then \eqref{riesz lambda} simplifies into:
\begin{equation}
\varphi(v)
=
\varphi(x_0)
+
\int_{\{v<X\le x_0\}}(X-v)d\lambda
+
\int_{\{x_0<X\le v\}}(v-X)d\lambda.
\end{equation}

The above result can be stated in a slightly different 
way:
\begin{corollary}
Let $X\in\Fun{\Omega}$ with 
$\cl{X[\Omega]}=\R$, 
$\varphi\in\Fun{\R}$. Define $x_0$ and 
$h_u^v$ as in Theorem \ref{th convex} 
and assume 
$D^+\varphi(x_0)=D^-\varphi(x_0)=0$. $\varphi$ 
is convex if and only if there exists a measure 
structure $(\Ring,\lambda)$ on $\Omega$ such that 
(a)
$\lambda(u<X<v)=0$ when 
$ D^+\varphi(v)
	\le 
D^-\varphi(u)$,
(b) 
$\{h_u^v(X):v\ge u\}\subset L^1(\lambda)$ and
\begin{equation}
\label{phi nu}
\varphi(v)
	=
\varphi(u)
	+
\int h_u^v(X)d\lambda
\qquad
v\ge u.
\end{equation}
\end{corollary}

\begin{proof}
Define
$
\Neg
	=
\big\{\{u<X<v\}:u,v\in\R, D^+\varphi(v)
	\le 
D^-\varphi(u)\big\}
$.
From $\cl{X[\Omega]}=\R$ follows that 
$\{u<X<v\}\in\Neg$ if and only if 
$D^+\varphi(v)\le D^-\varphi(u)$ and that 
$\Neg$ is an ideal of sets. Then \eqref{phi nu} 
follows from Theorem \ref{th convex}.\tiref{iii}.
\end{proof}

\section{Applications to Statistics and Probability.}
\label{sec applications}

Returning to the Bayesian problem described 
in the Introduction, fix a
$(\A,m)\in\Meas(\Omega)$.

\begin{theorem}\label{th bayes}
Let $X\in\Fun{\Omega,S}$. The following properties 
are equivalent:
(i)
there exist a family $\{Q_\theta:\theta\in\Theta\}$ 
of probabilities on $\A$ and an injective map
$G\in\Fun{\Theta,S}$ satisfying
\begin{subequations}
\begin{equation}\label{th conglo}
\int hdm<0
	\qtext{implies}
\inf_{\theta\in\Theta}\int hdQ_\theta<0
	\qquad
h\in\Sim(\A),
\end{equation}
\begin{equation}\label{Q*}
Q_\theta^*\big(A\cap\{X\ne G(\theta)\}\big)=0
\qquad
A\in\A,\ \theta\in\Theta;
\end{equation}
\end{subequations}
(ii)
there exist $K\in\Fun{\A\times S}$ and 
$(\Ring,\mu)\in\Meas(\Omega)$ such that
$\{K_s:s\in S\}\subset ba(\A)_+$ and, for 
each $A\in\A$, $E\subset S$ and $s\in S$,
\begin{subequations}
\begin{equation}
\label{kernel}
K(A, X)\in L^1(\mu)
\qtext{and}
m(A)
	=
\int K(A,X)d\mu,
\end{equation}
\begin{equation}\label{takeout}
A\cap\{X\in E\}\in\A(K_s) 
\qtext{and}
K\big(A\cap\{X\in E\};s\big)
	=
K\big(A;s\big)\set E(s).
\end{equation}
\end{subequations}
\end{theorem}

\begin{proof}
\timply{i}{ii}.
Since $G$ is injective we may define 
$K\in\Fun{\A\times S}$ by letting
\begin{equation}
\label{K}
K(A,s)
	=
Q_{G^{-1}(s)}(A)
\qquad
A\in\A,\ s\in G[\Theta]
\end{equation}
or $K(A,s)=0$ if $s\notin G[\Theta]$.
By \eqref{Q*},
$
\inf_\omega K(h;X(\omega))
	\le
\inf_\theta Q_\theta(h)
$
for every 
$
h\in\Sim(\A)
$
so that, letting $(Th)(\omega)=K(h;X(\omega))$ in 
Theorem \ref{th representation}, we conclude that 
$T$ is directed and $m$ is $T$-conglomerative. 
There exists then $(\Ring,\mu)\in\Meas(\Omega)$
such that
\begin{align*}
K(h,X)\in L^1(\mu)
\qtext{and}
\int hdm
	=
\int K(h,X)d\mu
	\qquad
h\in\Sim(A).
\end{align*}
If $A\in\A$ and $E\subset S$, then either 
$Q_\theta^*(A\cap\{X\in E\})=0$ 
(if $G(\theta)\notin E$) or 
$Q_\theta^*(A\cap\{X\in E^c\})=0$. In either case
$A\cap\{X\in E\}\in\A(K_s)\cap\A(Q_\theta)$ and
\begin{align*}
K\big(A\cap\{X\in E\};s\big)
	=
Q_{G^{-1}(s)}\big(A\cap\{X\in E\}\big)
	=
Q_{G^{-1}(s)}\big(A\cap\{X\in E\}\big)\set E(s)
	=
K(A,s)\set E(s).
\end{align*}
\timply{ii}{i}.
Take $S_0=\{s\in S:K_s\ne0\}$, $\Theta=S_0$, 
$Q_\theta=K_s$ and $G$ the identity. Then, 
\eqref{Q*} follows from \eqref{takeout}. To 
deduce \eqref{th conglo} from \eqref{kernel} 
it is enough to remark, via Theorem 
\ref{th companion}, that the identity on $S$ 
is trivially a companion to $X$ (relatively to 
the whole of $L^1(\mu)$).
\end{proof}

The kernel $K(A,s)$ in Theorem \ref{th bayes}
plays a prominent role in statistics in which it 
is interpreted as the prevision of $A$ 
conditional on the occurrence of $X=s$. Its 
existence is generally deduced from that of 
regular conditional expectation and requires 
some classical properties such as $S$ being
a Blackwell space. In Theorem \ref{th bayes}, 
instead, the existence of $K$ follows from 
$X$ strictly separating priors, so that each 
$\theta\in\Theta$ may be interpreted as a 
corresponding hypothesis concerning $X$. 

The following is an example of \eqref{Q*} in 
the classical setting.

\begin{example}
Let $X_1,X_2,\ldots$ be $m$-measurable random 
quantities on $\Omega$. Define implicitly the map
\begin{equation}
F(\omega,t)
	=
\lim_k\liminf_n\frac1n
\sum_{j=1}^n\sset{X_i\le t+2^{-k}}(\omega)
\end{equation}
of $\Omega$ into the set $\mathscr X$ of increasing, 
right continuous, $[0,1]$-valued functions on $\R$, 
the limiting empirical distribution. For each 
$\theta\in\Theta$, let $G(\theta)$ be a candidate 
distribution. In the classical case, with each 
$Q_\theta$ countably additive and the sequence 
$X_1,X_2,\ldots$ independently and identically 
distributed under each $Q_\theta$, condition 
\eqref{Q*}, with $X=F$, is a simple consequence 
of the strong law of large numbers, examined by 
Doob \cite{doob}. In order to guarantee that the 
inverse of $G$ is Borel measurable Doob assumes 
that $G$ is Borel measurable and that $\Theta$ 
is a subset of a complete and separable metric 
space, see also \cite{diaconis freedman}. 
\end{example}

We pass now to the classical problem of Skhorohod 
which has been studied by a number of authors too 
large to give exact references. We have been 
influenced by the work of Berti, Pratelli and Rigo 
\cite{berti pratelli rigo}. The starting point is the 
construction of a universal representation for the 
case of a separable space.

\begin{corollary}
\label{cor skhorohod}
Let $U\in\Fun{\Omega}$ with $\cl{U[\Omega]}$ 
having non empty interior and let $S$ be a separable, 
topological space. There exists a Borel function 
$H\in\Fun{\R,S}$ with countable range and such 
that $X'=H(U)$ is companion to any  pair $(X,m)$ 
relatively to $\Cts S$. 
\end{corollary}

\begin{proof}
By the remarks following Theorem \ref{th companion} 
we can assume with no loss of generality that $X$ is 
the identity. Given that $[a,b]\subset\cl{U[\Omega]}$ 
for some $a,b\in\R$ then, upon replacing $U$ with a 
suitable continuous transformation, we can assume 
that $\cl{U[\Omega]}=[0,1]$. Let $S_0$ be a 
countable, dense subset of $S$ and 
$\iota\in\Fun{\N,S_0}$ an enumeration of $S_0$. 
Define, 
\begin{equation}
\label{GH}
G(x)
	=
\inf\big\{n\in\N:1-2^{-n}\ge x\big\}
\qquad
x\in(0,1)
\quad\text{and}\quad
H
	=
\iota\circ G.
\end{equation} 
$H$ is a Borel function mapping $(0,1)$ onto $S_0$ 
-- since $G^{-1}({n})=(1-2^{-(n-1)},1-2^{-n}]$. 
If $h\in\Ha$ and $\int hdm<0$ then $\{h<0\}$ is 
an open, non empty subset of $S$ and as such it 
contains some element $\iota(n_h)$ of $S_0$. The 
set 
$
B_h
	=
\big\{U\in G^{-1}(n_h)\big\}
$
is non empty (as $\cl{U[\Omega]}=[0,1]$) and 
coincides with $\{X'=\iota(n_h)\}$. Thus, 
$B_h\subset\{h(X')<0\}$ so that $m$ is 
$X'$-conglomerative relatively to $\Ha$.
\end{proof}

Corollary \ref{cor skhorohod} extends to the case of 
finite additivity and of a separable state space the 
classical idea of generating a random quantity with 
given distribution by applying to a uniformly distributed 
random quantity the inverse of the corresponding 
cumulative density function. Interestingly, we obtain 
that the {\it same} function $X$ represents {\it all} 
possible distributions relatively to the class of 
continuous functions and for some suitable set 
function $\mu$. Let us also mention the possibility 
of dropping the condition that $S$ is separable by 
assuming that $m$ is supported by a measurable, 
separable subset of $S$.

We highlight the advantage of doing without 
measurability. Constructing a function such as 
$U$ in Corollary \ref{cor skhorohod} is a rather 
trivial exercise as long as $\Omega$ has the 
right cardinality. Requiring that $U$ is uniformly 
distributed on the unit interval under some 
classical probability measure $P$, as in the 
following Theorem \ref{th skhorohod}, requires, 
in contrast, additional assumptions. The following 
result is inspired by 
\cite[theorem 3.1]{berti pratelli rigo}.

\begin{theorem}
\label{th skhorohod}
Let $S$ be a normal, separable topological space, 
$\Sigma$ a ring of subsets of $S$ and 
$(\Omega,\A,P)$ a classical probability space 
supporting a random quantity $U$ uniformly 
distributed on $(0,1)$. Let either 
$m\in fa(\Sigma)_+$ be countably additive 
or $S$ be compact and write 
$\Ha=\Cts S\cap L^1(m)$. There exists a Borel 
function $g\in\Fun{(0,1),S}$ such that $X=g(U)$ 
is supported by $(\Omega,\A,P)$ and
\begin{equation}
\int hdm
	=
\int h(X)dP
\qquad
h\in\Ha.
\end{equation}
\end{theorem}

\begin{proof}
If $S$ is compact then the restriction of $m$ to the 
minimal ring $\Ring_\Ha$ is countably additive. Let 
$H$ be the map defined in \eqref{GH}. Then, as was 
shown in the proof of Corollary \ref{cor skhorohod}, 
$m$ is $H$-conglomerative relatively to $\Cts S$ so 
that, by Theorem \ref{th companion},
\begin{equation}
\int hdm
	=
\int h(H)d\mu
\qquad
h\in\Ha
\end{equation}
for some $(\Ring,\mu)\in\Meas((0,1))$. 
We claim that $\sigma\Ring=\Bor((0,1))$. Recall 
that $\sigma\Ring$ is generated by sets of the 
form $\{h(H)>t\}$ which are Borel since $h$ is 
continuous and $H$ is Borel. Conversely, if 
$0\le a\le b\le1$ then the set $H[(a,b)]$ is a 
finite subset of $S$ -- and therefore closed. 
Since $S$ is normal, for any other finite subset 
$F$ of $H[(a,b)^c]$ we can find a function 
$f\in\Fun{S,[0,1]}$ such that $f=1$ on $H[(a,b)]$ 
and $f=0$ on $F$. Thus 
$(a,b)
	\subset
\{f(H)\ge1\}\in\sigma\Ring$. 
Since $H[(0,1)]$ is countable we find a sequence 
$\seqn f$ of such functions each vanishing on a 
finite subset of $H[(a,b)^c]$ so that the intersection 
$\bigcap_n\{f_n(H)\ge1\}$ is again an element 
of $\sigma\Ring$ and coincides with $(a,b)$. 
In other words, we can assume that $\mu$ is 
defined on the Borel subsets of $(0,1)$. From 
the classical Skhorohod theorem, we deduce 
the existence of an $S$ valued random quantity 
$Z$ supported by $((0,1),\Bor((0,1)),\Lambda)$ 
(with $\Lambda$ the Lebesgue measure on $(0,1)$) 
and admitting $\mu$ as its distribution. On its 
turn, $\Lambda$ 
is the distribution of $U$ under $P$. A repeated 
application of Theorem \ref{th companion} with 
$g=H\circ Z$ and $X= g(U)$ gives
\begin{align*}
\int hdm
	=
\int h(H)d\mu
	=
\int h(g)d\Lambda
	=
\int h(X)dP
\qquad
h\in\Ha.
\end{align*}
Thus the random quantity $X$ is supported by 
$(\Omega,\A,P)$ and represents $m$ relatively 
to $\Cts S$.
\end{proof}

\section{Applications to stochastic processes.}

We start this section with a result closely related 
to Theorem \ref{th representation norm}.

\begin{theorem}
\label{th representation BM}
Let $\Ha\subset\Cts\R$ be a Stonean sublattice, 
$X'=(X'_t:t\in\R_+)$ Brownian motion on some, 
filtered, standard probability space $(\Omega',\A,P)$. 
Write $\Neg$ to denote the family of sets 
$A\subset\Omega\times\R_+$ such that
$P^*(\pi_\Omega A)=0$.
$\phi\in\Lin\Ha_+$ if and only if there exists a 
minimal $(\Ring,\mu)\in\Meas(\Omega'\times\R_+)$ 
with $\Neg\subset\Ring$, 
$\mu(N)=0$ for all $N\in\Neg$,
\begin{equation}
\label{representation BM}
h(X')\in L^1(\mu)
\qtext{and}
\phi(h)
=
\int h(X')d\mu
\qquad
h\in\Ha.
\end{equation}
Moreover, $\mu$ is countably additive if and only if
$\lim_n\phi(h_n)=0$ for every decreasing sequence
$\seqn h$ in $\Ha_+$ which converges to $0$ in
Lebesgue measure.
\end{theorem}

\begin{proof}
By Theorem \ref{th representation norm}, if 
$\phi(h)<0$ and $N\in\Neg$ then, since 
$N_t=\{\omega:(\omega,t)\in N\}$ is $P$
null
\begin{align*}
0
>
\inf_{\omega\in N_t^c}h(X'_t)(\omega)
\ge
\inf_{(\omega,s)\in N^c}h(X'_s)(\omega).
\end{align*}
The main claim follows immediately. The last
claim may be proved as in Theorem 
\ref{th representation norm} upon noting that
$X'[N^c]^c$ has $0$ Lebesgue measure when
$N\in\Neg$. But this is again clear since
$\{X'_t\in X'[N^c]^c\}$ is $P$ null. The rest of
that proof remains unchanged.
\end{proof}

Let $\mathcal I$ be the family of finite subsets 
of $\R_+$. For each 
$\alpha=\{t_1,\ldots,t_n\}\in\mathcal I$, 
let $\pi_\alpha$ be the projection
\begin{equation}
\pi_\alpha(s)=(s_{t_1},s_{t_2},\ldots,s_{t_n})
\qquad
s\in\Fun{\R_+}.
\end{equation}
If $X=(X_t:t\in\R_+)$,  write 
$X_\alpha=(X_t:t\in\alpha)$.

\begin{corollary}
\label{cor BM}
Let $X'=(X'_t:t\in\R_+)$ be Brownian motion on 
some classical probability space $(\Omega',\A,P)$
and $(m_\alpha:\alpha\in\mathcal I)$  a projective 
family of probabilities (namely 
$m_\alpha\in fa(\Bor(\R^\alpha))_+$ is the marginal 
of $m_\beta$ whenever $\alpha\subset\beta$). 
There exists a probability structure $(\A,\mu)$ on 
$\Omega$ such that 
\begin{equation}
h(X'_\alpha)\in L^1(\mu)
\qtext{and}
\int h dm_\alpha
	=
\int h(X'_\alpha)d\mu
\qquad
\alpha\in\mathcal I,\ 
h\in\Cts{\R^\alpha}\cap L^1(m_\alpha).
\end{equation}
If $m_\alpha$ is countably additive, then 
\begin{equation}
m_\alpha(B)
	=
\mu(X'_\alpha\in B)
\qquad
B\in\Bor(\R^\alpha).
\end{equation}
\end{corollary}

\begin{proof}
As usual, a projective family of probabilities 
induces a unique probability on the algebra
$\Sigma
=
\big\{\pi_\alpha^{-1}(B):\alpha\in\mathcal I,
B\in\Bor(\R^\alpha)\big\}$
of finite dimensional cylinders obtained by letting
\begin{equation}
m\big(\pi_\alpha^{-1}A\big)
	=
m_\alpha(A)
\qquad
A\in\Bor(\R^\alpha),\ 
\alpha\in\mathcal I.
\end{equation}
If $g\in\Fun{\R^\alpha}$ and 
$h=g\circ\pi_\alpha$ then
$
\{h>t\}
	=
\pi_\alpha^{-1}(\{g>t\})
$
so that from Lemma \ref{lemma L1} we conclude 
\begin{equation*}
\int h dm
	=
\int g dm_\alpha
\end{equation*}
whenever either side is well defined. Let 
$\Ha
	=
\{g\circ\pi_\alpha:
g\in\Cts{\R^\alpha},\ 
\alpha\in\mathcal I\}\cap L^1(m)
$. 
If $h\in\Ha$ and $\int hdm<0$, then 
$\int h_\alpha dm_\alpha<0$ 
for some 
$\alpha=\{t_1<\ldots<t_n\}\in\mathcal I$. 
Since
$\{h_\alpha<0\}$ is open and non empty, 
there exist open, non empty sets 
$B_1,\ldots,B_n\subset\R$ such that 
$x_i-x_{i-1}\in B_i$ for $i=1,\ldots,n$ (and $x_0=0$) 
implies $h_\alpha(x_1,\ldots,x_n)<0$. Therefore, 
$
P(X'_{t_1},\ldots,X'_{t_n}\in\{h_\alpha<0\})
\ge
\prod_{i=1}^nP(X'_{t_i}-X'_{t_{i-1}}\in B_i)
>
0
$
so that $\inf_\omega h_\alpha(X'_\alpha)<0$ and 
$m$ is $X'$-conglomerative. The second claim, as 
in Example \ref{ex normal}, follows from metric 
spaces being normal.
\end{proof}

Corollary \ref{cor BM} is related to 
\cite[Theorem 1]{dubins bizarre} 
and, in Dubins' peculiar terminology, it asserts that 
Brownian motion is {\it cousin} to any stochastic 
process. Dubins main finding is a necessary and
sufficient condition for the existence of cousins 
with almost all paths in a given class. His claim 
is an easy corollary of our previous results. We
give a simple proof for completeness.

\begin{corollary}[Dubins]
Let $X$ be a stochastic process on a probability 
space $(\Omega,\A,m)$ and let 
$\mathbb Y\subset\Fun{\R_+}$ satisfy:
\begin{equation}
\label{match}
\forall (\omega,\alpha)\in\Omega\times\mathcal I,
\quad 
\exists Y\in\mathbb Y
\qtext{such that}
Y(t)=X(\omega,t)
\qquad
t\in\alpha.
\end{equation}
There is a process $X'$ on a probability space 
$(\Omega',\Sigma,\mu)$ with $\mu$-a.a. paths in 
$
\mathbb Y
$ 
and such that
\begin{equation}
g(X_\alpha')\in L^1(\mu)
\qtext{and}
\int g(X_\alpha)dm
	=
\int g(X'_\alpha)d\mu
\qquad
\alpha\in\mathcal I,\ 
g\in\Fun{\R^\alpha},\ 
g(X_\alpha)\in L^1(m).
\end{equation}
\end{corollary}

\begin{proof}
Write
$$
\Ha
	=
\big\{g\circ\pi_\alpha:
\alpha\in\mathcal I,\ 
g\in\Fun{\R^\alpha},\ 
g(X_\alpha)\in L^1(m)\big\},
$$
$\Omega'=\Fun{\R_+}$ and define 
$T\in\bLin{\Ha,\Fun{\mathbb Y}}$ 
by letting 
$T(h)(Y)=h(Y)$ for each
$h\in\Ha$ and $Y\in\mathbb Y$. 
Then, $T$ is directed and, by \eqref{match}, 
the linear functional $\phi(h)=\int h(X)dm$ 
is $T$-conglomerative. By Theorem 
\ref{th representation} there exists a minimal 
$(\Ring_0,\mu_0)\in\Meas(\mathbb Y)$ such 
that
\begin{equation*}
\int h(X)dm
	=
\int hd\mu_0
\qquad
h\in\Ha.
\end{equation*}
Since $m$ is a probability, then $\Ring_0$ is an 
algebra and $\mu_0$ a probability. Let 
\begin{equation*}
\Sigma
	=
\big\{A\subset\Omega':A\cap\mathbb Y\in\Ring_0\big\}
\qtext{and}
\mu(A)
	=
\mu_0(A\cap\mathbb Y)
\qquad
A\in\Sigma.
\end{equation*} 
Then, $\Sigma$ is an algebra of subsets of $\Omega'$, 
$\mu$ a probability on $\Sigma$ with 
$\mu(\mathbb Y^c)=0$ and $X'(w,t)=w(t)$ a 
stochastic process on $(\Omega',\Sigma,Q)$ with 
$X'_w=w$. Moreover, 
$
\int hd\mu_0=\int h(X')d\mu
$
for all $h\in\Ha$.
\end{proof}

Dubins deduces from this result that any stochastic 
process admits \textit{cousins} having continuous 
or polynomial or stepwise linear paths.

\end{document}